\documentclass{article}

\usepackage{graphicx,xcolor,textpos}
\usepackage{helvet}
\usepackage{amssymb,latexsym,amsmath,epsfig,amsthm,hyperref,tipa,comment}
\newtheorem{iflargepthennotmonotone}{Theorem}
\newtheorem{ogbetter}[iflargepthennotmonotone]{Theorem}
\newtheorem{ohyeah1}[iflargepthennotmonotone]{Theorem}
\newtheorem{lima}[iflargepthennotmonotone]{Theorem}

\newtheorem{limc}[iflargepthennotmonotone]{Theorem}
\newtheorem{oldschool}[iflargepthennotmonotone]{Theorem}
\newtheorem{erniesquestion}[iflargepthennotmonotone]{Theorem}

\newtheorem{dirichlet}[iflargepthennotmonotone]{Theorem}
\newtheorem{godsgift}[iflargepthennotmonotone]{Theorem}
\newtheorem{Theoremonevtwo}[iflargepthennotmonotone]{Theorem}
\newtheorem{generallowertight}[iflargepthennotmonotone]{Theorem}
\newtheorem{powers}[iflargepthennotmonotone]{Theorem}
\newtheorem{nonperiodiclower}[iflargepthennotmonotone]{Theorem}
\newtheorem{lisl}{Lemma}
\newtheorem{test}[lisl]{Lemma}
\newtheorem{gcd1}[lisl]{Lemma}
\newtheorem{gabhigherpower}[lisl]{Lemma}
\newtheorem{neverodd}[lisl]{Lemma}
\newtheorem{pdoesntdivide}[lisl]{Lemma}

\newtheorem{pdoesntdivide3}[lisl]{Lemma}
\newtheorem{xbisxab}[lisl]{Lemma}
\newtheorem{primeorpower}[lisl]{Lemma}
\newtheorem{lcmbound}[lisl]{Lemma}

\newtheorem{bored}[lisl]{Lemma}
\newtheorem{bore}[lisl]{Lemma}

\newtheorem{pnotinsigma3}[lisl]{Lemma}
\newtheorem{pfreeinterval}[lisl]{Lemma}
\newtheorem{decreasingintervals}[lisl]{Lemma}
\newtheorem{diszero}[lisl]{Lemma}
\newtheorem{pdividesxab}[lisl]{Lemma}
\newtheorem{expo}[lisl]{Lemma}
\newtheorem{expolicit}[lisl]{Lemma}
\newtheorem{priemtellen}[lisl]{Lemma}
 
\newtheorem{carmichael}[lisl]{Lemma}
\newtheorem{notmuchleft}[lisl]{Lemma}
\newtheorem{dat}[lisl]{Lemma}
\newtheorem{baker}[lisl]{Lemma}
\newtheorem{kron}[lisl]{Lemma}
\newtheorem{gregmartin}[lisl]{Lemma}
\newtheorem{prodoverprimes}[lisl]{Lemma}
\newtheorem{gammaslim}[lisl]{Lemma}

\newtheorem{xpsolution}[lisl]{Lemma}
\newtheorem{xpnietklein}[lisl]{Lemma}
\newtheorem{pxabopnieuw}[lisl]{Lemma}
\newtheorem{jamaarhoezodan}[lisl]{Lemma}
\newtheorem{upperlinc}[lisl]{Lemma}
\newtheorem{lowerlinc}[lisl]{Lemma}
\newtheorem{cbounds}[lisl]{Lemma}
\newtheorem{tnogood}[lisl]{Lemma}
\newtheorem{polproperties}[lisl]{Lemma}
\newtheorem{frob}[lisl]{Lemma}

\newtheorem{even}[lisl]{Lemma}

\newtheorem{derange}[lisl]{Lemma}
\newtheorem{easyasabc}[lisl]{Lemma}
\newtheorem{largedlargesum}[lisl]{Lemma}

\newtheorem{pdivpowersum}[lisl]{Lemma}
\newtheorem{whendismultiple}[lisl]{Lemma}

\newtheorem{somebettercdvalues}[lisl]{Lemma}
\newtheorem{radineqs}[lisl]{Lemma}
\newtheorem{corollary}{Corollary}
\newtheorem{infinite}[corollary]{Corollary}
\newtheorem{finitedsjonnie}[corollary]{Corollary}
\newtheorem{stormer}[corollary]{Corollary}
\newtheorem{linnik}[corollary]{Corollary}
\newtheorem{somecdvalues}[corollary]{Corollary}
\newtheorem{vabepsilongeneral}[corollary]{Corollary}
\newtheorem{vabepsilonclassical}[corollary]{Corollary}

\begin{document}

\vspace*{-2cm}

\Large
 \begin{center}
On the non-monotonicity of the denominator of generalized harmonic sums  

\hspace{10pt}

\large
Wouter van Doorn \\

\hspace{10pt}

\end{center}

\normalsize

\vspace{-10pt}

\centerline{\bf Abstract}

Let $\displaystyle \sum_{i=a}^b \frac{1}{i} = \frac{u_{a,b}}{v_{a,b}}$ with $u_{a,b}$ and $v_{a,b}$ coprime. In their influential monograph \mbox{\cite[p. 34]{cnt}}, Erd\H os and Graham ask, among many other questions, the following: Does there, for every fixed $a$, exist a $b$ such that $v_{a,b} < v_{a,b-1}$? If so, what is the least such $b = b(a)$? In this paper we will investigate these problems in a more general setting, answer the first question in the affirmative and obtain the bounds $a + 0.54\log(a) < b(a) \le 4.374(a-1)$, which hold for all large enough $a$.

\tableofcontents

\section{Introduction} \label{intro}
\subsection{Introduction}
Let $\{r_i\}_{i \in \mathbb{N}}$ be a fixed periodic sequence of integers, not all equal to $0$, with period $t$. That is, for every $i \in \mathbb{N}$ we have $r_{i+t} = r_i$ and for at least one (and therefore for infinitely many) $i$, $r_i \neq 0$. For a given positive integer $a$, we shall be concerned with sums of the form $\displaystyle \sum_{i=a}^b \frac{r_i}{i}$. More precisely, if $u_{a,b} \in \mathbb{Z}$ and $v_{a,b} \in \mathbb{N}$ are coprime integers for which $\displaystyle \frac{u_{a,b}}{v_{a,b}} = \sum_{i=a}^b \frac{r_i}{i}$, we will be interested in whether $v_{a,b} < v_{a,b-1}$ holds for some $b$. \\

Paul Erd\H os and Ronald Graham asked this question in \mbox{\cite{cnt}} for the case where $r_i = 1$ for all $i$, and this was answered in the affirmative independently by Peter Shiu in \mbox{\cite{dhn}} and in unpublished work (predating the current manuscript) by the author. Even though the pre-print \mbox{\cite{dhn}} only explicitly deals with $a = 1$, their methods can be used for arbitrary $a \in \mathbb{N}$ as well. In personal communication Ernie Croot then asked about the far more general result where $r_i \in A$ for some fixed finite set $A$. This generalization turns out to be false, however. So it seems natural to ask for a reasonable condition on the $r_i$ that does guarantee that the inequality $v_{a,b} < v_{a,b-1}$ holds for some $b$, and it will turn out that periodicity is sufficient. \\

Note that, in common vernacular, $v_{a,b} < v_{a,b-1}$ means that the fraction was simplified. Since a fraction can be simplified precisely when both numerator and denominator share a prime divisor, we would like to get a handle on the prime factorizations of $u_{a,b}$ and $v_{a,b}$. However, even in the special case of the harmonic numbers $H_n$, where $r_i = 1$ for all $i$, $a = 1$ and $b = n$, surprisingly little is known about this. \\

For example, in \mbox{\cite{har1}} it was conjectured that for every prime $p$ the numerator of $H_n$ is only finitely often divisible by $p$, and this is still unsolved. In the other direction, we have a well-known eponymous theorem by Wolstenholme (\mbox{\cite{wol}}) stating that for any prime number $p \ge 5$, the numerator of $H_{p-1}$ is divisible by $p^2$. Various generalizations and extensions of this result are known and can be found in \mbox{\cite{wol2}}. Let $L_n$ be the least common multiple of $1,2,\ldots,n$. In \mbox{\cite{dhn}} Shiu shows that for every sequence of odd primes $p_1, p_2, .., p_k$ there exists a positive integer $n$ such that the denominator of $H_n$ is a divisor of $\frac{L_n}{p_1p_2 \cdots p_k}$, as long as the terms $\theta_i = \frac{\log(p_1)}{\log(p_i)}$ are rationally independent for $1 \le i \le k$. The latter is unfortunately not known for $k \ge 3$, although it would follow from conjectures like Schanuel's Conjecture. In the other direction it is often conjectured (see e.g. \mbox{\cite{cnt}}, \mbox{\cite{dhn}} and \mbox{\cite{chi0}}) that there exist infinitely many $n$ for which the denominator of $H_n$ is equal to $L_n$, and this too is not yet solved. \\

Even though here we will focus on the inequality $v_{a,b} < v_{a,b-1}$, in a series of papers (\mbox{\cite{chi0}}, \mbox{\cite{chi1}}, \mbox{\cite{chi2}}, \mbox{\cite{chi3}}, \mbox{\cite{chi4}}, \mbox{\cite{chi5}}, \mbox{\cite{chi6}}, \mbox{\cite{chi7}}, \mbox{\cite{chi8}}), Chen, Wu and Yan prove various results on the density of $b$ for which equality occurs. For example, in \mbox{\cite{chi5}} it is shown that, as long as $|r_i| = 1$ for all $i$, the density of $n$ for which $v_{1, n} = v_{1, n+1}$ is $1$. The same result is obtained in \mbox{\cite{chi6}}, for the case of $r_i \in \{0, 1\}$ with $r_i = 1$ if, and only if, $i \equiv k \pmod{t}$ for a specific residue class $k \pmod{t}$. 

\subsection{Overview of results}
The main theorem we obtain in Section \ref{upper} is that for every $a \in \mathbb{N}$ there exist infinitely many integers $b > a$ for which $v_{a,b} < v_{a,b-1}$. Furthermore, if we denote by $b(a)$ the smallest such $b$, then there exists an effective constant $c$, which only depends on the sequence $\{r_i\}_{i \in \mathbb{N}}$, such that $b(a) < ca$. For example, in the original case $r_1 = t = 1$ we have the upper bound $b(a) \le 4.374(a-1)$, which is true for all $a \ge 6$. \\

In Section \ref{lower} we will look at lower bounds and prove that $b(a) > a + (\frac{1}{2} - \epsilon)\log(a)$ holds for all $\epsilon > 0$ and all large enough $a$. This lower bound turns out to be close to optimal, because for $t > 1$ there are infinitely many $a$ with $b(a) < a + t^3\log(a)$. We may therefore deduce that the lower limit ${\displaystyle\liminf_{a \rightarrow \infty}} \left(\frac{b(a)-a}{\log a}\right)$ then exists and is bounded between $\frac{1}{2}$ and $t^3$. We can reduce $t^3$ to $20 \log(\log(2t))$ in the case where $r_i \neq 0$ for all $i$ with $\gcd(i, t) = 1$, and to $2$ if $r_i \neq 0$ for all $i$. We will end this section with even further improvements when $r_i = 1$ for all $i$, and show $0.54 < {\displaystyle\liminf_{a \rightarrow \infty}} \left(\frac{b(a)-a}{\log a}\right) < 0.61$ in that case. \\

In Section \ref{noperiod} we will consider two possible generalizations. First we will look at sums of the form $\displaystyle \frac{u_{a,b}}{v_{a,b}} = \sum_{i=a}^b \frac{r_i}{i^d}$, where $d$ is a positive integer, and we define $b_d(a)$ to be the smallest positive integer $b$ for which $v_{a,b} < v_{a,b-1}$. We will then show that, if at least two out of $r_1, r_2, r_3, r_4, r_5$ are non-zero and $d$ is large enough, then $b_d(a)$ is finite for all $a$. Afterwards, we will focus on the case where all $r_i$ are equal to $1$ and prove that there exists a constant $c_d = O\big(\log^{10}(d)\big)$ so that for every $a$, $b_d(a) \le c_da$. We will furthermore calculate this constant $c_d$ for all $d < 120$. Finally, we will look at what happens when the sequence $\{r_i\}_{i \in \mathbb{N}}$ is no longer assumed to be periodic. For example, if we only assume $r_i = \pm 1$, then it is possible that $v_{a,b}$ is a monotone increasing function of $b$. In fact, we will see that there are very few results in this paper that generalize to the non-periodic case. Two results that however do generalize, are the lower bound $b(a) > a + (\frac{1}{2} - \epsilon)\log(a)$, and a theorem stating that if the $r_i$ are non-zero and remain bounded, then a function similar to $u_{1,b}$ has arbitrarily large prime divisors.

\subsection{Notation and definitions}
Recall that $u_{a,b}$ and $v_{a,b} \ge 1$ are coprime integers with $\displaystyle \frac{u_{a,b}}{v_{a,b}} = \sum_{i=a}^b \frac{r_i}{i}$. Here, $r_1, r_2, \ldots$ is a given periodic sequence of integers, which are not all equal to $0$. The integer $a$ should be viewed as fixed, but arbitrary, and $b(a)$ denotes the smallest integer $b > a$ such that $v_{a,b} < v_{a,b-1}$. Instead of directly dealing with the sequence $v_{a,b}$ however, we shall instead work with the more robust sequence $L_{a,b}$, defined as the least common multiple of all integers $i \in \{a, a+1, .., b \}$ for which $r_i \neq 0$. We then define $X_{a,b}$ as $X_{a,b} = L_{a,b} \displaystyle \sum_{i=a}^b \frac{r_i}{i}$ and abbreviate $L_{1,n}$ and $X_{1,n}$ to $L_n$ and $X_n$ respectively. With $g_{a,b}$ defined as the greatest common divisor of $X_{a,b}$ and $L_{a,b}$, we get $v_{a,b} = \frac{L_{a,b}}{g_{a,b}}$. All of these values clearly depend on the sequence of $r_i$, and this dependence is always implicit; the sequence of $r_i$ should be viewed as fixed. \\

The letters $p$ and $q$ are reserved for prime numbers, $t$ will always refer to the period of the sequence of $r_i$, and most other (Roman) letters will generally denote integers, often non-negative. Whenever we say that $p^k$ \textit{exactly divides} an integer $n$, we mean that $n$ is divisible by $p^k$, but not by $p^{k+1}$. If the prime $p$ is fixed or understood, then $e(n)$ denotes the non-negative integer $k$ such that $p^k$ exactly divides $n$. If $p$ does not divide $n$ at all, then $e(n) = 0$, while $e(0) = \infty$ for all $p$. When confusion might arise, we will use a subscript like $e_p(n)$, to emphasize the dependence on the prime $p$. \\

$O\big(f(x)\big)$ and $o\big(f(x)\big)$ are the familiar Big-O and Little-o notations, while $x | y$ reads `$x$ divides $y$'. The symbols $\mathbb{R}$, $\mathbb{Z}$ and $\mathbb{N}$ represent the set of real numbers, the set of integers and the set of positive integers respectively. The greek letter $\lambda = \lambda(t)$ will be the Carmichael function; the smallest positive integer such that $p^{\lambda} \equiv 1 \pmod{t}$ for all $p$ with $\gcd(p, t) = 1$. The dependence of $\lambda$ on $t$ will always be implicit and we have $\lambda | \varphi(t)$, where $\varphi$ is Euler's totient function. The number of primes smaller than or equal to $n$ is denoted by $\pi(n)$, and we often make use of the prime number theorem which states ${\displaystyle \lim_{n \rightarrow \infty}} \frac{\pi(n) \log(n)}{n} = 1$. We will refer to both the prime number theorem and its generalization to arithmetic progressions by the acronym PNT. Finally, $\epsilon$ will denote a small, positive real number.

\newpage

\section{Upper bounds} \label{upper}
\subsection{Proof strategy}
Our goal in this section is to prove that $b(a)$ is finite and, moreover, that there exists a constant $c$ such that for every $a$ we have $b(a) < ca$. For pedagogical purposes we will first prove this in Section \ref{iflarge} by assuming the existence of a certain large prime divisor $p$ of $X_n$, for some $n \in \mathbb{N}$. This furthermore motivates the next step of the proof: trying to find such a large prime divisor. That such a prime exists is immediate when $r_1 = t = 1$, initially leading to a bound of $b(a) \le 6a$ in that case. In Section \ref{return} we will look at some examples and prove that when $r_i = 1$ and $a \ge 6$, we can tighten the bound to $b(a) \le 4.374(a-1)$. \\

To find this large prime divisor of $X_n$, we first have to show a lower bound on $X_n$ itself. We will do this in Section \ref{expgrowth} where we initially prove that there exists a constant $c_0$ such that $|X_n| > c_0^n$ holds for all large enough $n$. This follows from some estimates on $\frac{X_n}{L_n}$ and the fact that $L_n$ grows exponentially fast. However, in the end we not only would like to prove $b(a) < ca$, we actually want to give an explicit value for this constant $c$ as well. So phrases like `for large enough $n$' will generally not suffice. Therefore, we take some time to find an interval that we can write down explicitly, where $|X_n|$ is large enough for our purposes for sufficiently many $n$ in that interval. \\

Section \ref{largeprimedivisors} is then aimed at proving that the prime divisors of $X_n$ get arbitrarily large. If we let $r = \max_i |r_i|$ and define $m = 1 + \max(r,t)$ (although any integer larger than $\max(r,t)$ also works), then our proof will actually show that for every interval $I$ of length at least $e^{6m}$, there exists an $n \in I$ for which $X_n$ is divisible by a prime $p \ge m$. \\

To prove this, we split up the primes into three subsets $\Sigma_1$, $\Sigma_2$ and $\Sigma_3$. The first subset contains the primes larger than or equal to $m$, so it would suffice to find an $n \in I$ for which the largest divisor of $X_n$ containing only primes from $\Sigma_2$ or $\Sigma_3$ is smaller than $|X_n|$. Then we will see that the largest divisor of $X_n$ containing only primes from $\Sigma_3$ is always small in a certain congruence class. And finally, let $2 \le p_1 < p_2 < \ldots < p_y < m$ be the primes in $\Sigma_2$. We will construct a nesting sequence of intervals $I \supset I_1 \supset I_2 \supset \ldots \supset I_y$, for which the largest power of $p_{\sigma(j)}$ that divides $X_n$ is small for all $n \in I_j$, where $\sigma: \{1,2,\ldots,y\} \rightarrow \{1,2,\ldots,y\}$ is a permutation. And so for all $n \in I_y$ and all $p_j \in \Sigma_2$, the largest power of $p_j$ that divides $X_n$ is small. Combining these estimates on the powers of primes from $\Sigma_2$ and $\Sigma_3$ that divide $X_n$ then implies that $X_n$ must have a prime divisor from $\Sigma_1$ as well. \\

Write $n = lp^k$ with $\gcd(l, p) = 1$ and $p \ge m$ a prime that divides $X_n$. By Section \ref{largeprimedivisors} such $n$ and $p$ exist. Then by setting $b = np^{\lambda k_1}$ for some suitable $k_1$, it turns out that in order to show $v_{a,b} < v_{a,b-1}$, we need to check $\gcd(l, X_{a,b-1}) < p$. Now, in the case that $r_i \neq 0$ for all $i$ with $\gcd(i, t) = 1$, we have $l < p$, so this condition is trivially satisfied. This will allow us to calculate an explicit upper bound in Section \ref{dirichlet} for the constant $c$ for which $b(a) < ca$ holds for all $a$, when $\gcd(i, t) = 1$ implies $r_i \neq 0$. This $c$ turns out to grow doubly exponential in $m$. \\

In the general case it is possible that $l > p$, which makes it more difficult to check the condition $\gcd(l, X_{a,b-1}) < p$. So our goal is to make sure that $\gcd(l, X_{a,b-1})$ is small and we therefore need some information on the prime divisors of $l$ and $X_{a,b-1}$. Section \ref{primebound} is then dedicated to proving that for every prime $q \notin \Sigma_3$ there are intervals $I$ such that for all $n \in I$, $e_q(X_n)$ is small. \\

In Section \ref{dioph} we then pick a prime $q \notin \Sigma_3$ such that $e_q(l)$ is large. Using results from Section \ref{primebound} we can ensure that, if $b-1$ is contained in a certain interval, then $e_q(X_{a,b-1})$ is small. This makes $\gcd(l, X_{a,b-1})$ small as well, which accomplishes our goal. These intervals are of the form $[c_qq^{\lambda k_2}, (c_q + 1)q^{\lambda k_2})$, where $c_q$ is a constant and $k_2$ can be any integer. So when we now choose $b = np^{\lambda k_1}$, for some $k_1$, then we need the inequalities $c_qq^{\lambda k_2} < np^{\lambda k_1} \le (c_q + 1)q^{\lambda k_2}$ to hold. When we take logarithms, we end up with a linear form in logarithms and, using a well-known Diophantine approximation result by Dirichlet, these inequalities can be satisfied infinitely often, implying that $b(a)$ is finite. \\

Finally, by using an extension of a result by Baker, we also have a lower bound for the linear form in logarithms that we encountered in Section \ref{dioph}. In Section \ref{final} we then use this lower bound to give an explicit linear upper bound for $b(a)$. In this general case the constant $c$ grows triply exponential in $m$.

\subsection{Under the assumption of a large prime divisor} \label{iflarge}
Let $r = \max_i |r_i|$ and define $i_1$ to be the smallest positive integer such that $r_{i_1} \neq 0$. Now let $p > \max(r, t)$ be a prime number that divides $X_i$ for some integer $i \ge i_1$ and let $n = n(p)$ be the smallest such $i$. In Section \ref{largeprimedivisors} we will prove that such a prime $p$ actually exists, but for now we will simply assume we have one at our disposal. \\

Necessarily we see that $p$ does not divide $X_{n-1}$ and $r_n \neq 0$. Since $p > \max(r,t)$, this implies $0 < |r_n| < p$. Write $n = lp^{k}$ with $\gcd(l, p) = 1$ and recall that $\lambda$ is such that $q^{\lambda} \equiv 1 \pmod{t}$, whenever $\gcd(q, t) = 1$. Now we set $b = np^{\lambda k_1} = lp^{\lambda k_1 + k}$, where $k_1$ is an integer for which $p^{\lambda k_1 + k} \ge \max(a, 2t)$. We then have the following theorem.

\begin{iflargepthennotmonotone}\label{Theorem1} 
If $\gcd(l, X_{a,b-1}) < p$, then $v_{a,b} < v_{a,b-1}$. Furthermore, if the condition $\gcd(l, X_{a,b-1}) < p$ is satisfied for the smallest $k_1$ such that $p^{\lambda k_1 + k} \ge \max(a, 2t)$ holds, then $b(a) \le \max(a-1,2t-1)lp^{\lambda}$. 
\end{iflargepthennotmonotone}
 
\begin{proof}
Let us first remark that the second part can be quickly seen, because for the smallest possible $k_1$, we have $p^{\lambda (k_1 - 1) + k} \le \max(a-1, 2t-1)$, implying $b = lp^{\lambda k_1 + k} \le \max(a-1, 2t-1)lp^{\lambda}$. Now, recall that we in general have $v_{a,b} = \frac{L_{a,b}}{g_{a,b}}$. And thus, if $L_{a,b} = L_{a,b-1}$, then $v_{a,b} < v_{a,{b-1}}$ holds true, precisely when $g_{a,b} > g_{a,b-1}$. We claim that, indeed, $L_{a,b}$ and $L_{a,b-1}$ are equal while $g_{a,b}$ is larger than $g_{a,b-1}$. We start with the first part of this claim, but before we do so, we need some properties. 

\begin{test}\label{prelimprop}
There exists a positive integer $j$ with $1 \le j < l$ for which $r_{jp^k} \neq 0$. Furthermore, $p^k$ exactly divides $L_n$ and $p^{\lambda k_1 + k}$ exactly divides $L_{a,b}$.
\end{test}

\begin{proof}
As we will do a lot in this paper, we look at $X_n \pmod{p}$ and remove the terms in the sum which are divisible by $p$. Since $r_n \neq 0$, $L_n$ must be divisible by $n = lp^k$, and therefore by $p^k$. Therefore, if $i \in [1, n]$ is an integer such that $\frac{L_nr_i}{i}$ does not vanish modulo $p$, then $p^k$ divides $i$. Now assume by contradiction that $r_{jp^k}$ is equal to $0$ for all $j < l$. This implies in particular that $p^k$ exactly divides $L_n$. Moreover, there would only be one $i \in [1, n]$ for which $\frac{L_nr_i}{i}$ does not vanish modulo $p$, namely $i = lp^k = n$ itself. So by applying $0 < |r_n| < p$, we would then get the following:
\begin{align*}
X_n &= L_n \sum_{i=1}^n \frac{r_i}{i} &\\
&= \sum_{i=1}^n \frac{L_n r_i}{i} & \\
&\equiv \sum_{i=1}^l \frac{L_nr_{ip^k}}{ip^k} &\pmod{p} \\
&\equiv \frac{L_nr_n}{lp^k} &\pmod{p} \\
&\not\equiv 0 &\pmod{p}
\end{align*}

And this would contradict the assumption that $p$ divides $X_n$. So this proves the first property, which in turn implies that $jp^k$ and therefore $p^k$ divides $L_{n-1}$, so that $p$ does not divide $\frac{L_n}{L_{n-1}}$. \\

For the other two properties, recall that $L_n$ is divisible by $p^k$. Furthermore, since $b = np^{\lambda k_1} \equiv n \pmod{t}$, we see $r_b = r_n \neq 0$, which implies that $L_{a,b}$ is divisible by $p^{\lambda k_1 + k}$. To prove that these are also the largest powers of $p$ dividing $L_n$ and $L_{a,b}$, assume by contradiction that $p^{\lambda k_1 + k + 1}$ divides $L_{a,b}$. We will show that this implies that $L_n$ is divisible by $p^{k+1}$, which will lead to a contradiction. If $p^{\lambda k_1 + k + 1}$ divides $L_{a,b}$, then there exists a positive integer $g$ with $a \le g \le b$ such that $g$ is divisible by $p^{\lambda k_1 + k + 1}$ and $r_g \neq 0$. Now we can choose $h = gp^{-\lambda k_1} \le bp^{-\lambda k_1} = n$ and note that $h \equiv g \pmod{t}$ by definition of $\lambda$, so $r_h = r_g$, which we assumed to be non-zero. Furthermore, $h$ would be divisible by $p^{k+1}$ and, since $r_h \neq 0$, so would $L_n$. However, $\frac{L_n r_n}{n}$ would then vanish modulo $p$ and we would get $X_n = \frac{L_n}{L_{n-1}}X_{n-1} + \frac{L_n r_n}{n} \equiv \frac{L_n}{L_{n-1}}X_{n-1} \pmod{p}$. This is impossible, since it contradicts the assumption that $n$ is the smallest $i$ for which $p$ divides $X_i$. 
\end{proof}

We will now prove that $L_{a,b}$ and $L_{a,b-1}$ are equal to each other, in which case $v_{a,b} < v_{a,b-1}$ is equivalent with $g_{a,b} > g_{a,b-1}$. 

\begin{lisl}\label{lisl}
With $b = lp^{\lambda k_1 + k} \ge l\max(a, 2t)$, we get $L_{a,b} = L_{a,b-1}$.
\end{lisl} 

\begin{proof}
Since $L_{a,b} = \text{lcm}(b, L_{a,b-1}) = \text{lcm}(lp^{\lambda k_1 + k}, L_{a,b-1})$ with $\gcd(l, p^{\lambda k_1 + k}) = 1$, it suffices to show that both $l$ and $p^{\lambda k_1 + k}$ divide $L_{a,b-1}$. \\

We observe $l | (b - lt)$ and we claim that this implies $l | L_{a,b-1}$. To see this, first note $r_{b - lt} = r_b = r_n \neq 0$. Secondly, $b > b - lt \ge l\max(a, 2t) - lt = l\max(a - t,t) \ge 2\max(a - t,t) \ge a$, where $l \ge 2$ follows from Lemma \ref{prelimprop}. And so we conclude that $b - lt$, which is a multiple of $l$, is contained in the interval $[a,b-1]$ and must therefore divide $L_{a,b-1}$. \\

To show that $p^{\lambda k_1 + k}$ divides $L_{a,b-1}$, we use the existence of a positive integer $j < l$ for which $r_{jp^{k}} \neq 0$, as guaranteed by Lemma \ref{prelimprop}. We then see that $r_{jp^{\lambda k_1 + k}} \neq 0$ as well, while $a \le p^{\lambda k_1 + k} \le jp^{\lambda k_1 + k} < lp^{\lambda k_1 + k} = b$. And so $L_{a,b-1}$ is divisible by $jp^{\lambda k_1 + k}$, and in particular by $p^{\lambda k_1 + k}$. 
\end{proof}

Now it suffices to show $g_{a,b} > g_{a,b-1}$. Morally, this holds because $p | X_n$ implies $p | X_{a,b}$ as well.

\begin{pdividesxab} \label{p-ness}
The prime $p$ divides $X_{a,b}$, while $p$ does not divide $X_{a,b-1}$.
\end{pdividesxab}

\begin{proof} Let us take a look at $X_n \pmod{p}$ again.
\begin{align*}
X_n &= L_n \sum_{i=1}^n \frac{r_i}{i} & \\
&\equiv L_n \sum_{i=1}^l \frac{r_{ip^{k}}}{ip^{k}} &\pmod{p} \\
&\equiv \frac{L_n}{p^{k}} \sum_{i=1}^l \frac{r_{ip^{k}}}{i} &\pmod{p}
\end{align*}

By Lemma \ref{prelimprop}, $p^k$ exactly divides $L_n$, so for this final sum to be congruent to $0 \pmod{p}$ we must have $\displaystyle \sum_{i=1}^l \frac{r_{ip^{k}}}{i} \equiv 0 \pmod{p}$. Now let us use this knowledge in the analogous sum for $X_{a,b}$.
\begin{align*}
X_{a,b} &= L_{a,b} \sum_{i=a}^b \frac{r_i}{i} & \\
&\equiv L_{a,b} \sum_{i=1}^{l} \frac{r_{ip^{\lambda k_1 + k}}}{ip^{\lambda k_1 + k}} &\pmod{p} \\
&\equiv \frac{L_{a,b}}{p^{\lambda k_1 + k}} \sum_{i=1}^{l} \frac{r_{ip^{k}}}{i} &\pmod{p} \\
&\equiv 0 &\pmod{p}
\end{align*}

And indeed we see that $p$ divides $X_{a,b}$ as well. On the other hand, note that $p$ does not divide $\frac{L_{a,b}r_b}{lp^{\lambda k_1 + k}}$ by Lemma \ref{prelimprop}. From this observation the inequality $X_{a,b-1} = X_{a,b} - \frac{L_{a,b}r_b}{lp^{\lambda k_1 + k}} \not\equiv X_{a,b} \pmod{p}$ follows, and we conclude that $p$ does not divide $X_{a,b-1}$. 
\end{proof}

Now we are almost ready to finish up our proof, but before we do so, we need one last lemma.

\begin{gcd1}\label{Lemma1} 
For all primes $q$ we have $e_q(g_{a,b}) \ge e_q(g_{a,b-1}) - \min\big(e_q(X_{a,b-1}), e_q(b)\big)$.
\end{gcd1}

\begin{proof}
Let us fix the prime $q$ for this proof. From $e(X_{a,b-1}) \ge e(g_{a,b-1})$ the inequality $e(g_{a,b}) \ge e(g_{a,b-1}) - e(X_{a,b-1})$ immediately follows, since $e(g_{a,b})$ is non-negative. It therefore suffices to show $e(g_{a,b}) \ge e(g_{a,b-1}) - e(b)$.
\begin{align*}
e(g_{a,b}) &=\min \big(e(X_{a,b}), e(L_{a,b})\big) \\
&= \min\left(e\left(\frac{L_{a,b}}{L_{a,b-1}}X_{a,b-1} + \frac{L_{a,b}r_b}{b}\right), e(L_{a,b})\right) \\
&\ge \min\left(e(X_{a,b-1}), e\left(\frac{L_{a,b}r_b}{b}\right), e\left(L_{a,b-1}\right) \right) \\
&\ge \min\big(e(X_{a,b-1}), e\left(L_{a,b-1}\right) \big) - e(b) \qedhere
\end{align*}
\end{proof}

We will now calculate $g_{a,b}$ to finish the proof of Theorem \ref{Theorem1}.
\begin{align*}
g_{a,b} &= \prod_{q \text{ prime}} q^{e_q(g_{a,b})} \\
&= p^{e_p(g_{a,b})} \prod_{q | l} q^{e_q(g_{a,b})} \prod_{q \nmid b} q^{e_q(g_{a,b})}  \\
&\ge p^{e_p(g_{a,b-1})+1} \prod_{q | l} q^{e_q(g_{a,b-1}) - \min(e_q(X_{a,b-1}), e_q(b))} \prod_{q \nmid b} q^{e_q(g_{a,b-1})}  \\
&= \frac{p}{\gcd(l, X_{a,b-1})} g_{a,b-1} \\
&> g_{a,b-1} \qedhere
\end{align*}
\end{proof}

\subsection{Some examples and a return to the classical case} \label{return}
Since $\gcd(l, X_{a,b-1}) \le l \le n$, it is worth pointing out that as soon as we find an integer $n$ and a prime $p > \max(r,t,n)$ such that $p$ divides $X_n$, then the condition in Theorem \ref{Theorem1} is satisfied and $b(a)$ is finite for all $a$. In practice in turns out that, regardless of the sequence $r_1, r_2, \ldots$ that is chosen, one very often quickly finds such positive integers $n$ for which $X_n$ is divisible by a prime $p > \max(r, t, n)$. As an instructive example, let us look at all possible sequences of $r_i$ for which $\max(r, t) \le 2$. \\

Without loss of generality we assume that the first non-zero $r_i$ is positive, and for $t = 2$ we may assume $r_1 \neq r_2$. With these assumptions there are $12$ distinct sequences with $\max(r, t) \le 2$. We have tabulated these sequences, together with an $n$ and a prime $p > \max(r, t, n)$ such that $X_n$ is divisible by $p$. 

\begin{center}
  \begin{tabular}{| l | l | l | l | l |}
		\hline
    $t$ & $r_1$ & $r_2$ & $n$ & $p$ \\ \hline \hline
    1 & 1 & - & 2 & 3 \\ \hline
    1 & 2 & - & 2 & 3 \\ \hline
    2 & 1 & -2 & 2 & 3 \\ \hline
		2 & 1 & -1 & 3 & 5 \\ \hline
    2 & 1 & 0 & 7 & 11 \\ \hline
		2 & 1 & 2 & 3 & 7 \\ \hline
    2 & 2 & -2 & 3 & 5 \\ \hline
    2 & 2 & -1 & 2 & 3 \\ \hline
		2 & 2 & 0 & 7 & 11 \\ \hline
    2 & 2 & 1 & 2 & 5 \\ \hline
		2 & 0 & 1 & 6 & 11 \\ \hline
    2 & 0 & 2 & 6 & 11 \\ \hline
  \end{tabular}
\end{center}  

By extending this table with the help of a computer, one can check that for all sequences of $r_i$ with $\max(r, t) \le 8$, there exist $n$ and $p$ with $\max(r, t, n) < p \le 179$ and $p | X_n$. With these $n$ and $p$ we can then apply Theorem \ref{Theorem1}. For example, for all $12$ tabulated sequences we get the upper bound $b(a) \le 77a$, for all $a \ge 3$. In particular, if $r_i = 1$ for all $i$, we obtain the following corollary of Theorem \ref{Theorem1}:

\begin{corollary}\label{classical}
If $r_i = 1$ for all $i$, then $b(a) \le 6(a-1)$, for all $a > 1$.\footnote{See \mbox{\cite{oeis1}} for the actual values of $b(a) - 1$.}
\end{corollary}

It is however possible to improve upon this corollary. Recall that, if $k$ is such that $3^k < a \le 3^{k+1}$, then the proof of Theorem \ref{Theorem1} shows that with $f(a) = 2 \cdot 3^{k+1}$ one has $v_{a,f(a)} < v_{a,f(a)-1}$. So for all $a \in (3^k, 3^{k+1}]$ the same value of $f(a)$ is chosen. To improve upon Corollary \ref{classical}, for $k \ge 10$ we are going to split up the interval $(3^k, 3^{k+1}]$ into six sub-intervals and let the value of $f(a)$ depend on the sub-interval that contains $a$. First, let us state our improvement.

\begin{ogbetter} \label{ogplus}
If $r_i = 1$ for all $i$, then $b(a) \le 4.374(a-1)$, for all $a \ge 6$.
\end{ogbetter}

\begin{proof}
To prove this, we will define a function $f(a)$ for all $a \ge 6$ such that $f(a) \le 4.374(a-1)$ and $v_{a,f(a)} < v_{a,f(a)-1}$. To start off, for $6 \le a \le 59049 = 3^{10}$, we define $f(a)$ as in the following four tables, where elements in the top rows specify intervals of $a$. \\

\begin{tabular}{| l || c | c | c | c | c | c | c | c | c |}
  \hline			
  $a$ & $[6, 10]$ & $[11, 14]$ & $[15, 27]$ & $[28, 50]$ & $[51, 81]$ & $[82, 108]$ & $[109, 117]$ \\
	\hline
  $f(a)$ & $15$ & $35$ & $54$ & $75$ & $162$ & $135$ & $126$ \\
  \hline  
\end{tabular}

\vspace{5pt}

\begin{tabular}{| l || c | c | c | c | c |}
  \hline			
  $a$ & $[118, 243]$ & $[244, 363]$ & $[364, 729]$ & $[730, 1000]$ & $[1001, 2187]$ \\
	\hline
  $f(a)$ & $486$ & $968$ & $1458$ & $2166$ & $4374$  \\
  \hline  
\end{tabular}

\vspace{5pt}

\begin{tabular}{| l || c | c | c | c | c |}
  \hline			
  $a$ & $[2188, 2916]$ & $[2917, 3000]$ & $[3001, 6561]$ & $[6562, 8748]$ & $[8749, 9000]$ \\
	\hline
  $f(a)$ & $3645$ & $3402$ & $13122$ & $10935$ & $10206$\\
  \hline  
\end{tabular}

\vspace{5pt} 
\begin{tabular}{| l || c | c | c | c |}
  \hline			
  $a$  & $[9001, 19683]$ & $[19684, 26244]$ & $[26245, 27000]$ & $[27001, 59049]$\\
	\hline
  $f(a)$ & $39366$ & $32805$ & $30618$ & $118098$\\
  \hline  
\end{tabular}

\vspace{10pt}

With these values of $f(a)$, one can check that $f(a) \le 4.374(a-1)$ holds for all $a \le 3^{10}$ and with the help of a computer, one can also check $v_{a,f(a)} < v_{a,f(a) - 1}$ in each case, proving Theorem \ref{ogplus} for all $a \le 3^{10}$. \\

We may therefore assume $a > 3^{10}$, in which case there exists an integer $k \ge 10$ such that $3^k < a \le 3^{k+1}$. We will now partition the interval $I = (3^k, 3^{k+1}]$ into the following six subintervals: 
\begin{align*}
I_1 &= (3^k, 10\cdot 3^{k-2}] \\
I_2 &= (10\cdot 3^{k-2}, 11\cdot 3^{k-2}] \\
I_3 &= (11\cdot 3^{k-2}, 4\cdot 3^{k-1}] \\
I_4 &= (4\cdot 3^{k-1}, 37\cdot 3^{k-3}] \\
I_5 &= (37\cdot 3^{k-3}, 1000\cdot 3^{k-6}] \\
I_6 &= (1000\cdot 3^{k-6}, 3^{k+1}]
\end{align*}

We then define $f(a)$ as follows:
$$
f(a) = \begin{cases} 

5\cdot3^{k-1} &\mbox{if } a \in I_1 \\ 
16\cdot3^{k-2} &\mbox{if } a \in I_2 \\ 
5\cdot3^{k-1} &\mbox{if } a \in I_3 \\
14\cdot3^{k-2} &\mbox{if } a \in I_4 \\
1024\cdot3^{k-6} &\mbox{if } a \in I_5 \\
2\cdot3^{k+1} &\mbox{if } a \in I_6 \\

\end{cases} 
$$

The inequality $f(a) \le 4.374(a-1)$ is again straight-forward to check for all $a \in I$. It therefore suffices to prove $v_{a,f(a)} < v_{a,f(a) - 1}$. For all $a \in I_6$, the proof of Theorem \ref{Theorem1} tells us $v_{a,f(a)} < v_{a,f(a) - 1}$. For $a$ in the other five intervals, Theorem \ref{Theorem1} does not directly help, but we will follow its proof quite closely with $p = 3$. 

First, analogously to Lemma \ref{lisl}, we remark that in all cases $L_{a,f(a)} = L_{a,f(a)-1}$. To see this, write $f(a) = l \cdot 3^{k_1}$ with $\gcd(l, 3) = 1$, and recall that $L_{a,f(a)}$ equals $\text{lcm}(l \cdot 3^{k_1}, L_{a,f(a)-1})$. Since $l$ divides $l(3^{k_1} - 1)$, $3^{k_1}$ divides $(l-1)3^{k_1}$ and, in all cases, $a \le \min(l(3^{k_1} - 1), (l-1)3^{k_1})$, we get $L_{a,f(a)} = \text{lcm}(l \cdot 3^{k_1}, L_{a, f(a)-1}) = L_{a,f(a)-1}$. It therefore suffices to show $g_{a,f(a)} < g_{a,f(a)-1}$. \\

The main difference with the proof of Theorem \ref{Theorem1} is that here, $X_{a,f(a)}$ is not just divisible by $3$; we actually claim that $9$ divides $X_{a,f(a)}$ for all $a$ in the first four intervals, while $27 | X_{a,f(a)}$ for $a \in I_5$. We will then make use of the following result, which can be obtained by going through the computation of $g_{a,b}$ again, at the end of the proof of Theorem \ref{Theorem1}.

\begin{gabhigherpower} \label{higherpower}
If $\gcd(l, X_{a,f(a)-1}) < p^{e_p(g_{a,f(a)}) - e_p(g_{a,f(a)-1})}$, then $g_{a,f(a)} < g_{a,f(a)-1}$.
\end{gabhigherpower}

To show that $9$ (or $27$) does indeed divide $X_{a,f(a)}$ for $a \in \displaystyle \bigcup_{1 \le i \le 5} I_i$, we use the fact that if $e_3(L_{a,f(a)}) = k_1$, then $\frac{L_{a,f(a)}}{i} \equiv 0 \pmod{3^m}$, unless $e_3(i) > k_1 - m$. So to calculate $X_{a,f(a)} \pmod{3^m}$ the only terms $\frac{L_{a,f(a)}}{i}$ that we have to add are the ones where $3^{k_1 - m + 1}$ divides $i$. Note that in all the five intervals we consider, we have $3^k < a < f(a) < 2 \cdot 3^k$, so that $k_1$ is at most $k-1$.

\begin{enumerate}
	\item For $a \in I_1$ we chose $f(a) = 5 \cdot 3^{k-1}$, so that $e_3(L_{a,f(a)})$ equals $k-1$. This means that, modulo $9$, the only terms $\frac{L_{a,f(a)}}{i}$ that are non-zero, are the ones where $i$ is divisible by $3^{k-2}$. We will now calculate $X_{a,f(a)} \pmod{9}$ by rearranging those terms and then taking certain pairs of terms together.
\begin{align*}	
X_{a,f(a)} &= L_{a,f(a)} \sum_{i=a}^b \frac{1}{i} \\
&\equiv \sum_{i=10}^{15} \frac{L_{a,f(a)}}{i \cdot 3^{k-2}} \pmod{9} \\
&= \frac{L_{a,f(a)}}{3^{k-1}} \Biggr[ 3\left(\frac{1}{10} + \frac{1}{11} \right) + \left(\frac{1}{4} + \frac{1}{5}\right) + 3\left(\frac{1}{13} + \frac{1}{14}\right) \Biggr] \\
&= \frac{L_{a,f(a)}}{3^{k-1}} \Biggr[ 9\left(\frac{7}{10 \cdot 11} \right) + 9\left(\frac{1}{4 \cdot 5}\right) + 9\left(\frac{9}{13 \cdot 14}\right) \Biggr] \\
&\equiv 0 \pmod{9}
\end{align*}

\item For $a \in I_2$, we also have $e_3(L_{a,f(a)}) = k-1$, and we obtain the following sum:
\begin{align*}	
X_{a,f(a)} &\equiv \sum_{i=11}^{16} \frac{L_{a,f(a)}}{i \cdot 3^{k-2}} \pmod{9} \\
&= \frac{L_{a,f(a)}}{3^{k-1}} \Biggr[ 3\left(\frac{1}{11} + \frac{1}{16} \right) + \left(\frac{1}{4} + \frac{1}{5}\right) + 3\left(\frac{1}{13} + \frac{1}{14}\right) \Biggr] \\
&= \frac{L_{a,f(a)}}{3^{k-1}} \Biggr[ 9\left(\frac{9}{11 \cdot 16} \right) + 9\left(\frac{1}{4 \cdot 5}\right) + 9\left(\frac{9}{13 \cdot 14}\right) \Biggr] \\
&\equiv 0 \pmod{9}
\end{align*}	

\item The calculation for $a \in I_3$ is very similar to the one for the first interval, except that it does not contain the two terms corresponding to $10 \cdot 3^{k-2}$ and $11\cdot3^{k-2}$.
\begin{align*}	
X_{a,f(a)} &\equiv \sum_{i=12}^{15} \frac{L_{a,f(a)}}{i \cdot 3^{k-2}} \pmod{9} \\
&= \frac{L_{a,f(a)}}{3^{k-1}} \Biggr[ \left(\frac{1}{4} + \frac{1}{5}\right) + 3\left(\frac{1}{13} + \frac{1}{14}\right) \Biggr] \\
&= \frac{L_{a,f(a)}}{3^{k-1}} \Biggr[ 9\left(\frac{1}{4 \cdot 5}\right) + 9\left(\frac{9}{13 \cdot 14}\right) \Biggr]  \\
&\equiv 0 \pmod{9}
\end{align*}

\item For $a \in I_4$ we have that $e_3(L_{a,f(a)})$ is equal to $k-2$ and $\frac{L_{a,f(a)}}{i} \equiv 0 \pmod{9}$, unless $e_3(i) \ge k-3$.	
\begin{align*}
X_{a,f(a)} &\equiv \sum_{i=37}^{42} \frac{L_{a,f(a)}}{i \cdot 3^{k-3}} \pmod{9}\\
&= \frac{L_{a,f(a)}}{3^{k-2}} \Biggr[ 3\left(\frac{1}{37} + \frac{1}{38} \right) + \left(\frac{1}{13} + \frac{1}{14}\right) + 3\left(\frac{1}{40} + \frac{1}{41}\right) \Biggr] \\
&= \frac{L_{a,f(a)}}{3^{k-2}} \Biggr[ 9\left(\frac{25}{37 \cdot 38} \right) + 9\left(\frac{3}{13 \cdot 14}\right) + 9\left(\frac{27}{40 \cdot 41}\right) \Biggr] \\
&\equiv 0 \pmod{9}
\end{align*}	

\item Finally, for $a \in I_5$, $e_3(L_{a,f(a)}) = k-4$ and $\frac{L_{a,f(a)}}{i} \equiv 0 \pmod{27}$, unless $e_3(i) \ge k-6$.	Since $999 \cdot 3^{k-6} < a \le 1000 \cdot 3^{k-6} < 1024 \cdot 3^{k-6} = f(a)$, this means that there are in total $25$ terms which do not vanish modulo $27$. We partition those $25$ terms into eight pairs of the form $\frac{L_{a,f(a)}}{i \cdot 3^{k-6}} + \frac{L_{a,f(a)}}{(i+1) \cdot 3^{k-6}}$ where $i \equiv 1 \pmod{3}$, then three more pairs of the form $\frac{L_{a,f(a)}}{(999 + i) \cdot 3^{k-6}} + \frac{L_{a,f(a)}}{(1026 - i) \cdot 3^{k-6}}$ where $i$ is divisible by $3$ but not by $9$, and then the three remaining terms. We claim that the sum of every pair is divisible by $27$, and so is the sum of the three remaining terms.
\begin{align*}	X_{a,f(a)} &\equiv \sum_{i=1000}^{1024} \frac{L_{a,f(a)}}{i \cdot 3^{k-6}} \pmod{27} \\[10pt]
&= \frac{L_{a,f(a)}}{3^{k-6}} \Biggr[ \sum_{j = 0}^{7} \left(\frac{1}{1000 + 3j} + \frac{1}{1001 + 3j} \right)  \\
&\hspace{7pt}+ \left(\frac{1}{1002} + \frac{1}{1023}\right) + \left(\frac{1}{1005} + \frac{1}{1020}\right) + \left(\frac{1}{1011} + \frac{1}{1014}\right)   \\
&\hspace{7pt}+ \left(\frac{1}{1008} + \frac{1}{1017} + \frac{1}{1024}\right) \Biggr]  \\[10pt]
&= \frac{L_{a,f(a)}}{3^{k-4}} \Biggr[27 \sum_{j = 0}^{7} \left(\frac{667 + 2j}{(1000 + 3j)(1001 + 3j)} \right)  \\
&\hspace{7pt}+ 27\left(\frac{75}{334 \cdot 341}\right) + 27\left(\frac{75}{335 \cdot 340}\right) + 27\left(\frac{75}{337 \cdot 338}\right)   \\
&\hspace{7pt}+ 27\left(\frac{797}{7 \cdot 113 \cdot 1024}\right) \Biggr] \\[10pt]
&\equiv 0 \pmod{27} 
\end{align*}
\end{enumerate}

For $a \in I_1 \cup I_3 \cup I_4$, we see $e_3\left(\frac{L_{a,f(a)}}{f(a)}\right) = 0$. For $a \in I_2$, we have $e_3\left(\frac{L_{a,f(a)}}{f(a)}\right) = 1$. And for $a \in I_5$, $e_3\left(\frac{L_{a,f(a)}}{f(a)}\right) = 2$. Since $X_{a,f(a)-1} = X_{a,f(a)} - \frac{L_{a,f(a)}}{f(a)}$, this implies (compare with Lemma \ref{p-ness}) the following (in)equalities:
\begin{align*}
e_3(X_{a,f(a)-1}) &= 0 \le e_3(X_{a,f(a)}) - 2 \text{ for } a \in I_1 \cup I_3 \cup I_4 \\
e_3(X_{a,f(a)-1}) &= 1 \le e_3(X_{a,f(a)}) - 1 \text{ for } a \in I_2 \\
e_3(X_{a,f(a)-1}) &= 2 \le e_3(X_{a,f(a)}) - 1 \text{ for } a \in I_5
\end{align*}

Since $L_{a,f(a)}$ is always at least $k-4 \ge 3$, it suffices by Lemma \ref{higherpower} to show $\gcd(l, X_{a,f(a)-1}) < 9$ for $a \in I_1 \cup I_3 \cup I_4$ and $\gcd(l, X_{a,f(a)-1}) < 3$ for $a \in I_2 \cup I_5$. Since $l = 5, 16, 5, 14, 1024$ for $I_1, I_2, I_3, I_4, I_5$ respectively, this at once follows from the following well-known proposition. \phantom\qedhere

\begin{neverodd} \label{neverodd}
If $r_i = 1$ for all $i$, then $X_{a,b}$ is odd for all $a$ and $b \ge a$.
\end{neverodd}

\begin{proof}
Let $m$ be such that $L_{a,b}$ is exactly divisible by $2^m$, and let $i \in [a,b]$ be an integer divisible by $2^m$. Then we claim that this $i$ is unique; if $i' \neq i$ is also divisible by $2^m$, then $i' \notin [a,b]$. To see this, first note that if $i'$ is divisible by $2^m$, then either $i' \le i - 2^m$ or $i' \ge i + 2^m$. Secondly note that, since $i$ is exactly divisible by $2^m$, it must be an odd multiple of $2^m$. This implies that $i - 2^m$ and $i + 2^m$ are both even multiples of $2^m$, which means they are divisible by $2^{m+1}$. Since $L_{a,b}$ is not divisible by $2^{m+1}$, this then shows that both $i - 2^m$ and $i + 2^m$ have to be outside of the interval $[a,b]$, so $i'$ cannot be contained in $[a,b]$ either. Since we have shown that this $i$ is unique, we conclude $X_{a,b} \equiv \frac{L_{a,b}}{i} \equiv 1 \pmod{2}$.
\end{proof}
\end{proof}

\subsection{Exponential growth} \label{expgrowth}
In Section \ref{iflarge} we used a prime $p > \max(r, t)$ that divides $X_n$, for some $n \in \mathbb{N}$. We will now start to concern ourselves with proving the existence of such a prime. In order to do this, the first thing we need to find are lower bounds on the growth of $X_n$ itself. For whomever just wants an exponential lower bound that works for all large enough $n$, we will prove that first. However, in this paper we aim for explicit bounds, and for that we need to work a bit harder, which we shall do right after. 

\begin{lcmbound} \label{lcmbound}
For all $n \ge t(t+2)$ we have $L_n > 2^{\frac{n}{t} -2}$.
\end{lcmbound}

\begin{proof}
Recall that $i_1$ is the smallest positive integer such that $r_{i_1} \neq 0$, and define $A = \left\lfloor\frac{n - i_1}{t}\right\rfloor > \frac{n}{t} - 2$. We then have the following:
\begin{align*}
L_n &\ge \text{lcm}(i_1, i_1 + t, i_1 + 2t, \ldots, i_1 + At) \\
&\ge \text{lcm}\left(\frac{i_1}{\gcd(i_1, t)}, \frac{i_1 + t}{\gcd(i_1,t)}, \frac{i_1 + 2t}{\gcd(i_1,t)}, \ldots, \frac{i_1 + At}{\gcd(i_1, t)}\right)
\end{align*}

We can then apply Theorem $1.1$ from \mbox{\cite[p. 2]{lcm}}\footnote{With $\alpha = 1$, their $n$ is our $A$, their $r$ is our $\frac{t}{\gcd(i_1,t)}$ and their $u_0$ is our $\frac{i_1}{\gcd(i_1,t)}$.} to obtain a lower bound on $L_n$.
\begin{align*}
L_n &\ge \left(\frac{i_1}{\gcd(i_1,t)}\right)\left(\frac{t}{\gcd(i_1,t)}\right)\left(\frac{t}{\gcd(i_1,t)} + 1\right)^{A} \\
&> 2^{\frac{n}{t} -2}
\end{align*}

This lower bound holds when $A > \frac{t}{\gcd(i_1,t)}$. And if $n \ge t(t+2)$, then $A > \frac{n}{t} - 2 \ge t \ge \frac{t}{\gcd(i_1,t)}$.
\end{proof}

We will now use Lemma \ref{lcmbound} to prove a lower bound on $|X_n|$.

\begin{expo}
There exists a positive constant $c_0$ such that $|X_n| > c_0n^{-t}2^{\frac{n}{t}}$, for all large enough integers $n$.
\end{expo}

\begin{proof}
Fix a residue class $n \pmod{t}$ and note that the difference $\frac{X_{n+t}}{L_{n+t}} - \frac{X_n}{L_n}$ is equal to the sum $\displaystyle \sum_{i=n+1}^{n+t} \frac{r_i}{i}$ and can therefore be written as $\frac{f(n)}{g(n)}$, where $f(n)$ and $g(n)$ are non-zero polynomials with integer coefficients and degree at most $t$. If the leading coefficients of $f(n)$ and $g(n)$ have the same sign, then $\frac{f(n)}{g(n)}$ is positive for all large $n$, and if the leading coefficients of $f(n)$ and $g(n)$ differ in sign, then $\frac{f(n)}{g(n)}$ is negative for all large $n$. Either way, this implies that the sequence $\frac{X_n}{L_n}, \frac{X_{n+t}}{L_{n+t}}, \frac{X_{n+2t}}{L_{n+2t}}, ..$ is monotonic, for large enough $n$. If this sequence does not converge to zero, we are done by Lemma \ref{lcmbound}. If it does converge to zero, we have (for some positive constant $c$ and large enough $n$):
\begin{align*}
\left|\frac{X_n}{L_n}\right| & = \left|\frac{X_n}{L_n} - 0\right| \\
&> \left|\frac{X_n}{L_n} - \frac{X_{n+t}}{L_{n+t}}\right| \\
&= \left|\frac{f(n)}{g(n)}\right| \\
&> cn^{-t}
\end{align*}

We can now take $c_0$ to be the minimum value of $\frac{c}{4}$ over all residue classes modulo $t$, and we are once again done by Lemma \ref{lcmbound}.
\end{proof}

Like we mentioned before however, we would like to find explicit bounds. And to that end, we introduce some notation. Define $m = \max(r+1 , t+1)$ and note that by the table in Section \ref{return}, we may assume $m \ge 4$. Let $z$ be the number of primes strictly below $m$ and define $\tilde{m}$ to be the smallest integer larger than $42m^{3z+7}$ with $\tilde{m} \equiv i_1 \pmod{t}$ and such that $\tilde{m}$ has a prime divisor $q_0$ larger than $m^{3z+5}$. Finally, we define the half-open interval $I = [\tilde{m} - m^{3z+5}, \tilde{m} + m^{3z+5})$ and divide it into the sub-intervals $J_1 = [\tilde{m} - m^{3z+5}, \tilde{m})$ and $J_2 = [\tilde{m}, \tilde{m} + m^{3z+5})$. We can then show a lower bound on $|X_n|$ for all $n \in J_1$, or for all $n \in J_2$.


\begin{expolicit} \label{xpo}
Either $|X_n| > m^2n^{z}$ for all $n \in J_1$, or $|X_n| > m^2n^{z}$ for all $n \in J_2$.
\end{expolicit} 

\begin{proof}
Without loss of generality we may assume that there exists an integer $w \in J_1$ with $|X_w| \le m^2w^{z} < w^{z+1}$. Let $w+k$ be an integer in $J_2$ and note that $k$ is smaller than $(\tilde{m} + m^{3z+5}) - (\tilde{m} - m^{3z+5}) = 2m^{3z+5}$. We will then prove $|X_{w+k}| > (w+k)^{z+1} > m^2(w+k)^{z}$, but we first need a few technical lemmas. 

\begin{priemtellen} \label{priemteller}
For all $m \ge 2$ we have $z \le \pi(m) < \left(\frac{m}{\log(m)}\right)\min\left(1.25506, 1 + \frac{3}{2 \log(m)} \right)$. In particular, $m^{3z} < e^{3.77m}$ and $m^{3z} < e^{m \left(3 + \frac{9}{2 \log(m)} \right)}$. 
\end{priemtellen}

\begin{bore}\label{bla}
For all $k \in \mathbb{N}$ with $w+k \in J_2$ we have the following lower bound:
\begin{equation*}
\left|\sum_{i = w + 1}^{w + k} \frac{r_i}{i}\right| \ge \frac{1}{(w+k)^k} 
\end{equation*}
\end{bore}

\begin{bored}\label{boring}
For all $k \in \mathbb{N}$ with $w+k \in J_2$ we have the following inequality: 
\begin{equation*} 
\frac{2^{\frac{w+k}{t}-2}}{(w+k)^k} - (w+k)^k  w^{z+1} > (w+k)^{z+1}
\end{equation*}
\end{bored}

\begin{proof}[Proof of Lemma \ref{priemteller}]
These are the statements of Theorem $1$ and Corollary $1$ of \mbox{\cite[p. 69]{pnt2}}.
\end{proof}

\begin{proof}[Proof of Lemma \ref{bla}]
The sum $\displaystyle \sum_{i = w + 1}^{w + k} \frac{r_i}{i}$ can be written as a fraction with denominator equal to $L_{w+1,w+k}$, which is trivially upper bounded by $(w+k)^k$. So to prove that the estimate we want to show holds, it suffices to show that the left-hand side is non-zero. Note that $\tilde{m} \le w + k < \tilde{m} + m^{3z+5} < \tilde{m} + q_0$. So in the sum $\displaystyle \sum_{i = w + 1}^{w + k} \frac{L_{w+1,w+k}r_i}{i}$, every term is divisible by $q_0$, except for the term corresponding to $i = \tilde{m}$. The term corresponding to $i = \tilde{m}$ is not divisible by $q_0$ as $0 < |r_{\tilde{m}}| < q_0$. Since the sum is then not divisible by $q_0$, it is certainly non-zero, which means $\displaystyle \sum_{i = w + 1}^{w + k} \frac{r_i}{i}$ is non-zero as well.
\end{proof}

\begin{proof}[Proof of Lemma \ref{boring}]
We calculate, using the fact that $\frac{x}{\log(x)}$ is an increasing function of $x$ for $x \ge 3$, applying the inequalities $w+k \ge \tilde{m} > 42m^{3z+7}$ and $m^{3z} < e^{3.77m}$, and making use of the bounds $m \ge \max(4,z+1)$ and $4m^{3z+5} > 2k$.
\begin{align*}
\frac{w+k}{\log(w+k)} &> \frac{42m^{3z+7}}{\log(42m^{3z+7})} \\
&> \frac{42m^{3z+7}}{\log(42m^{7}e^{3.77m})} \\
&= \frac{42m^{3z+7}}{\log(42) + 7\log(m) + 3.77m} \\
&> \frac{42m^{3z+7}}{7.14m} \\
&> 5.88m^{3z+6} \\
&> 5m^2 + 5.8m^{3z+6} \\
&> 3t + 2(z+1)t + \frac{4tm^{3z+5}}{\log(2)} \\
&> \frac{3t}{\log(w+k)} + \frac{(z+1)t}{\log(2)} + \frac{2kt}{\log(2)}
\end{align*}

When we multiply by $\log(w+k)$, subtract $2t$ from both sides, then divide by $t$ and take $2$ to the power of both sides, we obtain:
\begin{align*}
2^{\frac{w+k}{t} - 2} &> 2(w+k)^{2k + z + 1} \\
&> (w+k)^{2k + z + 1} + (w+k)^{k + z + 1} \\
&> (w+k)^{2k} w^{z + 1} + (w+k)^{k + z + 1} 
\end{align*}

Dividing by $(w+k)^k$ and rearranging gives the desired inequality. 
\end{proof}

Combining all these lemmas lets us finish the proof of Lemma \ref{xpo}.
\begin{align*}
|X_{w+k}| &= \left|L_{w + k} \sum_{i=1}^{w + k} \frac{r_i}{i}\right| \\
&= \left|\frac{L_{w + k}}{L_{w}} X_{w} + L_{w + k} \sum_{i = w + 1}^{w + k} \frac{r_i}{i}\right| \\
&\ge L_{w + k} \left|\sum_{i = w + 1}^{w + k} \frac{r_i}{i}\right| - \frac{L_{w + k}}{L_{w}} \left|X_{w}\right| \\
&> \frac{2^{\frac{w+k}{t}-2}}{(w+k)^k} - (w+k)^k w^{z+1} \\ 
&> (w+k)^{z+1} \qedhere
\end{align*}
\end{proof}

\subsection{Large prime divisors exist}\label{largeprimedivisors}
With the notation of Lemma \ref{xpo}, set $I_0 = J_1$ if $|X_n| > m^2n^{z}$ holds true for all $n \in J_1$, or else set $I_0 = J_2$. This section will then be devoted to proving the following theorem.

\begin{ohyeah1}\label{ohyah} 
There exists an integer $n \in I_0$ for which $X_n$ is divisible by a prime larger than or equal to $m$. 
\end{ohyeah1}

Let $\Sigma_1, \Sigma_2, \Sigma_3$ be three sets of primes, defined as follows: 
\begin{enumerate}
	\item $\Sigma_1 = \{p: p \ge m\}$
	\item $\Sigma_2 = \{p: p < m, \text{ and } r_{ip^{e(t)}} \neq 0 \text{ for some } i\}$
	\item $\Sigma_3 = \{p: p < m, \text{ and } r_{ip^{e(t)}} = 0 \text{ for all } i\}$
\end{enumerate}

We will prove Theorem \ref{ohyah} by finding an $n \in I_0$ for which the largest divisor of $X_n$ that is composed solely of primes from $\Sigma_2 \cup \Sigma_3$ is strictly smaller than $|X_n|$. Let us start by focusing our attention on the primes from $\Sigma_3$ and note that, since $r_{ip^{e(t)}} = 0$ for all $i$, $p$ must divide $t$. Because otherwise, $e(t)$ would by assumption equal $0$, which would imply $r_i = 0$ for all $i$. To state and prove the following two lemmas, let us define $f_p = e(t) + e(r_{i_1})$.

\begin{diszero}\label{diszero}
If $p \in \Sigma_3$, then for all $n \in \mathbb{N}$ and all $i \in \mathbb{N}$ with $i + tp^{f_p} \le n$ we have $\dfrac{L_nr_i}{i} \equiv \dfrac{L_nr_{i + tp^{f_p}} }{i +tp^{f_p}} \pmod{p^{f_p}}$.
\end{diszero}

\begin{proof}
When $r_i = r_{i + tp^{f_p}} = 0$, Lemma \ref{diszero} follows immediately. We may therefore assume $r_i \neq 0$, which by definition of $\Sigma_3$ implies $e(i) < e(t)$. We can therefore define $i'$, $t'$ and $L'_n$ as $\frac{i}{p^{e(i)}}$, $\frac{t}{p^{e(i)}}$, and $\frac{L_n}{p^{e(i)}}$ respectively. Now the residue class $i' \equiv i' + t'p^{f_p} \pmod{p^{f_p}}$ is invertible, since $p$ does not divide $i'$. Let $i^*$ be its inverse. We then get the following:
\begin{align*}
\frac{L_nr_i}{i} - \frac{L_nr_{i + tp^{f_p}}}{i + tp^{f_p}} &= \frac{L'_nr_i}{i'} - \frac{L'_nr_{i + tp^{f_p}}}{i' + t'p^{f_p}} & \\
&\equiv L'_nr_ii^* - L'_nr_ii^* &\pmod{p^{f_p}} \hspace{30pt} \\
&\equiv 0 &\pmod{p^{f_p}} \hspace{30pt} \qedhere
\end{align*}
\end{proof}

For $p \in \Sigma_3$ we can use Lemma \ref{diszero} to bound $e_p(X_n)$, if $n$ lies in a certain residue class.

\begin{pdoesntdivide3}\label{Lemma3c}
If $p \in \Sigma_3$ and $n \equiv i_1 \pmod{t^3r_{i_1}^2}$, then $e_p(X_n) < f_p$.
\end{pdoesntdivide3}

\begin{proof}
If $n$ is congruent to $i_1 \pmod{t^3r_{i_1}^2}$, then a non-negative integer $c$ exists such that $n = i_1 + ctp^{2f_p}$. Now, by Lemma \ref{diszero} we know that $\frac{L_nr_i}{i}$ and $\frac{L_nr_{i + tp^{f_p}}}{i + tp^{f_p}}$ differ by a multiple of $p^{f_p}$. We can use this to split up the sum $X_n$ into parts that are all congruent modulo $p^{f_p}$. Writing $x_j = i_1 + jtp^{f_p}$, this yields:
\begin{align*}
X_n &= \sum_{i=1}^n \frac{L_n r_i}{i} &\\
&= \frac{L_n r_{i_1}}{i_1} + \sum_{j = 0}^{cp^{f_p} - 1} \sum_{i=x_j + 1}^{x_{j+1}} \frac{L_n r_i}{i} & \\
&\equiv \frac{L_n r_{i_1}}{i_1} + cp^{f_p} \sum_{i=x_0 + 1}^{x_{1}} \frac{L_n r_i}{i} &\pmod{p^{f_p}} \\
&\equiv \frac{L_n r_{i_1}}{i_1	}  &\pmod{p^{f_p}} \\
&\not\equiv 0  &\pmod{p^{f_p}}
\end{align*}

Here, the final inequality holds because $L_n$ is (regardless of the value of $n$) not divisible by $p^{e(t)}$, by the definition of $L_n$ and the assumption $p \in \Sigma_3$. 
\end{proof}

Note that Lemma \ref{Lemma3c} implies that for $n \equiv i_1 \pmod{t^3r_{i_1}^2}$, the largest divisor of $X_n$ composed solely of primes from $\Sigma_3$ is smaller than ${\displaystyle \prod_{p \in \Sigma_3}} p^{f_p} \le tr_{i_1} < m^2$. \\

Assume for the moment $n \in I_0$ and $n \equiv i_1 \pmod{t^3r_{i_1}^2}$. Since $|X_n| > m^2n^{z}$ by Lemma \ref{xpo} and since the largest divisor of $X_n$ composed solely of primes from $\Sigma_3$ is smaller than $m^2$, it follows that if the largest divisor of $X_n$ composed solely of primes from $\Sigma_2$ is smaller than $n^z$, then $X_n$ must have a prime divisor from $\Sigma_1$, which is exactly what we want. \\

So let $p_1 < p_2 < \ldots < p_y < m$ be the sequence of primes in $\Sigma_2$ with $y \le z$ and let (by a slight change in notation) $e_i(x)$ for the rest of this section denote the largest power of $p_{i}$ that divides $x$, where $e_i(0) = \infty$. With this notation, $p_1^{e_1(X_n)} \cdots p_{y}^{e_y(X_n)}$ is the prime decomposition of the largest divisor $d(n)$ of $X_n$ which consists only of primes contained in $\Sigma_2$. The goal is to find an $n \in I_0$ with $d(n) < n^z$. We define $m_0$ to be the smallest integer in $I_0$ that is congruent to $i_1 \pmod{t^3r_{i_1}^2}$ and note that such an integer $m_0 \in I_0$ exists, since $|I_0| = m^{3z + 5} >  m^5 > t^3r_{i_1}^2$. \\

We shall then construct a sequence $m_0 = n_1 < n_2 < \ldots < n_{y+1}$ of integers contained in $I_0$, such that $n_j \equiv i_1 \pmod{t^3r_{i_1}^2}$ for all $j$, and such that either $d(n_j) < n_j^y \le n_j^z$ for some $j \le y$, or $p_i^{e_i(X_{n_{y+1}})} < n_{y+1}$ for all $i$ with $1 \le i \le y$, implying $d(n_{y+1}) < n_{y+1}^y \le n_{y+1}^z$.

\begin{proof}[Proof of Theorem \ref{ohyah}]
To start off, choose $n_1 = m_0$. Now, once we have defined $n_j$ for some $j$ with $1 \le j \le y$, if $d(n_j) < n_j^y$, we are done, Theorem \ref{ohyah} is proved and we can stop. So for the rest of this proof we are free to assume that, after we have defined $n_j$, the inequality $d(n_j) \ge n_j^y$ holds. This implies in particular that there exists a $\sigma(j) \in \{1, 2, \ldots, y\}$ with $p_{\sigma(j)}^{e_{\sigma(j)}(X_{n_j})} \ge n_j$. Of course, there can be more than one such prime. Just pick, say, the smallest. \\

Then let $p_{\sigma(j)}^{k_j}$ be the largest power of $p_{\sigma(j)}$ smaller than $m^{3y + 6 - 3j}$, set $\widetilde{n}_{j+1}$ equal to the smallest integer larger than $n_j$ such that $e_{\sigma(j)}(\widetilde{n}_{j+1}) - e_{\sigma(j)}(r_{\widetilde{n}_{j+1}}) \ge k_j$, set $n_{j+1}$ equal to the smallest integer larger than or equal to $\widetilde{n}_{j+1}$ congruent to $i_1 \pmod{t^3r_{i_1}^2}$, and define the half-open interval $I_j = [n_{j+1}, n_{j+1} + p_{\sigma(j)}^{k_j} - m^{5})$. Then we claim that the intervals $I_j$ form a decreasing sequence. 

\begin{decreasingintervals} \label{decint}
We have $I_0 \supset I_1 \supset I_2 \supset \ldots \supset I_y$.
\end{decreasingintervals}

\begin{proof}
Since $I_j = [n_{j+1}, n_{j+1} + p_{\sigma(j)}^{k_j} - m^5)$ for $j \ge 1$ and $m^5$ is just a constant independent of $j$, we note that the statement $I_{j-1} \supset I_{j}$ for $j \ge 2$ is equivalent to the following two inequalities:
\begin{align*}
n_j &\le n_{j+1} \\
n_{j+1} + p_{\sigma(j)}^{k_j} &\le n_j + p_{\sigma(j-1)}^{k_{j-1}} 
\end{align*} 

While for $I_0 \supset I_1$ the second inequality gets replaced by $n_{2} + p_{\sigma(1)}^{k_1} - m^5 < \min(I_0) + m^{3z + 5}$, where $\min(I_0)$ is the smallest integer in $I_0$. And since $n_1 = m_0 < \min(I_0) +  m^5$, for $I_0 \supset I_1$ it suffices to prove $n_{2} + p_{\sigma(1)}^{k_1} \le n_1 + m^{3z + 5}$. \\

So we would like to get some upper and lower bounds on $n_{j+1}$ and $p_{\sigma(j)}^{k_j}$, and all we need to use are their definitions. First of all, as $n_{j+1}$ is defined as the smallest integer \textit{larger} than or equal to $\widetilde{n}_{j+1}$ for which something holds, while $\widetilde{n}_{j+1}$ is defined as the smallest integer \textit{larger} than $n_j$ with some property, the inequality $n_j \le n_{j+1}$ is trivial. \\

Secondly, for an upper bound on $n_{j+1}$, we need a small lemma.

\begin{pnotinsigma3} \label{notin3}
If $p \notin \Sigma_3$ and $A \in \mathbb{N}$ is such that $\gcd(A, t) = p^{e(t)}$, then for every $i \in \mathbb{N}$, there is an $i' \in \{iA, (i+1)A, \ldots, (i + \frac{t}{p^{e(t)}} - 1)A\}$ for which $r_{i'} \neq 0$.
\end{pnotinsigma3}

\begin{proof}
There are exactly $\frac{t}{p^{e(t)}}$ distinct residue classes $i' \pmod{t}$ that are divisible by $p^{e(t)}$, and all of them are represented in $\{iA, (i+1)A, \ldots, (i + \frac{t}{p^{e(t)}} - 1)A\}$, since $j_1A \equiv j_2A \pmod{t}$ implies $j_1 \equiv j_2 \pmod{\frac{t}{p^{e(t)}}}$. For at least one of them we must have $r_{i'} \neq 0$, by the assumption $p \notin \Sigma_3$.
\end{proof}

With $\mu_p$ defined as $\left \lfloor \frac{\log(m-1)}{\log(p)} \right \rfloor$, we apply Lemma \ref{notin3} with $A = p_{\sigma(j)}^{\mu_{p_{\sigma(j)}} + k_j}$. We then conclude that there exists an $i' \in (n_j, n_j + mA]$ with $r_{i'} \neq 0$ and $e_{\sigma(j)}(i') \ge \mu_{p_{\sigma(j)}} + k_j$. Now we recall that $\widetilde{n}_{j+1}$ is defined as the smallest integer larger than $n_j$ with $e_{\sigma(j)}(\widetilde{n}_{j+1}) - e_{\sigma(j)}(r_{\widetilde{n}_{j+1}}) \ge k_j$. And because $e_{\sigma(j)}(r_{\widetilde{n}_{j+1}}) \le \mu_{p_{\sigma(j)}}$ if $r_{\widetilde{n}_{j+1}} \neq 0$, we deduce $\widetilde{n}_{j+1} \le i' \le n_j + mA \le n_j +(m-1)mp_{\sigma(j)}^{k_j}$. And since, by definition of $n_{j+1}$, $n_{j+1} < \widetilde{n}_{j+1} + m^5$, we get $n_{j+1} < n_j + (m-1)mp_{\sigma(j)}^{k_j} + m^5$. \\

Lastly, we look for bounds on $p_{\sigma(j)}^{k_j}$. Again we have a trivial bound $p_{\sigma(j)}^{k_j} < m^{3y + 6 - 3j}$ because $p_{\sigma(j)}^{k_j}$ is defined as the largest power of $p_{\sigma(j)}$ \textit{smaller} than $m^{3y + 6 - 3j}$. On the other hand, there is always a power of $p_{\sigma(j)}$ between two consecutive powers of $m$ since $p_{\sigma(j)} < m$. So $p_{\sigma(j)}^{k_j}$ must be larger than $m^{3y + 5 - 3j}$. By putting all these inequalities together, we can prove $I_{j-1} \supset I_j$, for all $j \in \{2, \ldots, y\}$: 
\begin{align}
n_{j+1} + p_{\sigma(j)}^{k_j} &< n_j + (m-1)mp_{\sigma(j)}^{k_j} + m^5 + p_{\sigma(j)}^{k_j} \nonumber \\
&< n_j + (m-1)mp_{\sigma(j)}^{k_j} + 2p_{\sigma(j)}^{k_j} \nonumber \\
&\le n_j + m^2p_{\sigma(j)}^{k_j} \nonumber \\
&< n_j + m^{3y + 8 - 3j} \label{eqn8} \\
&= n_j + m^{3y + 5 - 3(j-1)} \nonumber \\
&< n_j + p_{\sigma(j-1)}^{k_{j-1}} \nonumber
\end{align}

To prove $I_0 \supset I_1$, use the above reasoning up to and including equation (\ref{eqn8}) with $j = 1$, and apply $y \le z$. 
\end{proof}

\begin{pfreeinterval} \label{pfree}
For all $n \in I_j$ we have $p_{\sigma(j)}^{e_{\sigma(j)}(X_{n})} < n$.
\end{pfreeinterval}

\begin{proof}
For an integer $n \in I_j$, let us write $X_n$ as a sum of four distinct terms.
\begin{align*}
X_n &= L_n \sum_{i=1}^n \frac{r_i}{i} \\
&= \sum_{i=1}^{n_j} \frac{L_nr_i}{i} + \sum_{i=n_j+1}^{\widetilde{n}_{j+1}-1} \frac{L_n	r_i}{i} + \frac{L_nr_{\widetilde{n}_{j+1}}}{\widetilde{n}_{j+1}} + \sum_{i=\widetilde{n}_{j+1}+1}^n \frac{L_nr_i}{i}
\end{align*}

By assumption, $X_{n_j}$ is divisible by a power of $p_{\sigma(j)}$ that is at least as large as $n_j$, hence we obtain $e_{\sigma(j)}\left(\frac{L_nX_{n_j}}{L_{n_j}}\right) \ge e_{\sigma(j)}(L_n) \ge e_{\sigma(j)}(L_n) -  k_j + 1$ for the first term. \\

As for the second and third terms, by the definition of $\widetilde{n}_{j+1}$ we know that for every $i \in [n_j+1, \widetilde{n}_{j+1}-1]$ we have $e_{\sigma(j)}\left(\frac{L_n	r_i}{i}\right) \ge e_{\sigma(j)}(L_n) - k_j + 1$, while $e_{\sigma(j)}\left(\frac{L_n	r_{\widetilde{n}_{j+1}}}{\widetilde{n}_{j+1}}\right) \le e_{\sigma(j)}(L_n) - k_j$. \\

Finally, since $e_{\sigma(j)}(\widetilde{n}_{j+1}) \ge k_j$ and $n < \widetilde{n}_{j+1} + p_{\sigma(j)}^{k_j}$, we have $e_{\sigma(j)}(i) < k_j$ for all $i \in [\widetilde{n}_{j+1}+1, n]$, hence $e_{\sigma(j)}\left(\frac{L_n	r_i}{i}\right) \ge e_{\sigma(j)}(L_n) - k_j + 1$. \\

Combining the above estimates we see that there is exactly one term in the sum for $X_n$ that is not divisible by $p_{\sigma(j)}^{e_{\sigma(j)}(L_n) - k_j + 1}$, and we conclude $p_{\sigma(j)}^{e_{\sigma(j)}(X_{n})} \le p_{\sigma(j)}^{e_{\sigma(j)}(L_n) - k_j} < n$.
\end{proof}

Now we may finish the proof of Theorem \ref{ohyah}. First off, all the $p_{\sigma(j)}$ have to be distinct, since $p_{\sigma(i)}^{e_{\sigma(i)}(X_{n_i})} \ge n_i$, while Lemma \ref{pfree} shows that if $i > j$, then for all $n \in I_{i-1} \subset I_j$ it holds true that $p_{\sigma(j)}^{e_{\sigma(j)}(X_{n})} < n$. In other words, $\big(\sigma(1), \sigma(2), \ldots, \sigma(y)\big)$ is a permutation of $(1,2,\ldots,y)$. Secondly, since our intervals form a nesting sequence, for $n_{y+1} \in I_y \subset I_j$ we have $p_{\sigma(j)}^{e_{\sigma(j)}(X_{n_{y+1}})} < n_{y+1}$ for all $j$ with $1 \le j \le y$. We conclude $d(n_{y+1}) = \displaystyle \prod_{j=1}^y p_j^{e_j(X_{n_{y+1}})} = \displaystyle \prod_{j=1}^y p_{\sigma(j)}^{e_{\sigma(j)}(X_{n_{y+1}})} < \displaystyle \prod_{j=1}^y n_{y+1} = n_{y+1}^y$, and the theorem is proved.
\end{proof}

\subsection{Explicit bounds for non-zero sequences and Dirichlet characters}\label{dirichlet}
Let $n \ge i_1$ be the smallest positive integer for which $X_n$ is divisible by a prime $p \ge m$. Write $n = lp^k$ and let $b$ be defined as in Section \ref{iflarge}. If we could force $\gcd(l, X_{a,b-1})$ to be smaller than $p$ (as is the condition in Theorem \ref{Theorem1}), then we can straightaway combine Theorem \ref{Theorem1} and Theorem \ref{ohyah}. We claim that this can be done when $r_i \neq 0$ for all $i$ with $\gcd(i, t) = 1$. Because in that case, it is not hard to see that $l$ will always be smaller than $p$, so the condition $\gcd(l, X_{a,b-1}) < p$ is fulfilled automatically. Indeed, by Lemma \ref{prelimprop} $p^k$ exactly divides $L_n$. But if $l > p$, then $n = lp^k > p^{k+1}$, while $r_{p^{k+1}} \neq 0$, so $p^{k+1}$ should divide $L_n$ as well; contradiction. \\

Recall $I_0 \subset I$ in the proofs of Lemma \ref{xpo} and Theorem \ref{ohyah}, where $I$ was defined as $[\tilde{m} - m^{3z+5} , \tilde{m} + m^{3z+5})$ and $\tilde{m}$ is the smallest integer larger than $42m^{3z+7}$ with $\tilde{m} \equiv i_1 \pmod{t}$ and such that $\tilde{m}$ has a prime divisor larger than $m^{3z+5}$. To find an upper bound on $\tilde{m}$ we use the results mentioned in the introduction of \mbox{\cite{lpfap}}, which imply $\tilde{m} < 43m^{3z+7}$. Now Theorem \ref{ohyah} implies $n < 43m^{3z+7} + m^{3z+5} < 44m^{3z+7} < e^{6m}$. Since $p$ divides $X_n$, $L_n < e^{1.04n}$ by Theorem $12$ in \mbox{\cite[p. 71]{pnt2}}, and $1.04 \cdot 44 < 46$, we can find an upper bound on $p$.
\begin{align*}
p &\le |X_n| \\
&\le L_n \sum_{i=1}^n \frac{|r_i|}{i} \\
&< 3 m \log(n) L_n \\
&< 18m^2 e^{46m^{3z+7}} \\
&< e^{47m^{3z+7}} 
\end{align*}

Now we can bound the quantity $\max(a-1,2t-1)lp^{\lambda}$ that appears in Theorem~\ref{Theorem1}.
\begin{align*}
\max(a-1,2t-1)lp^{\lambda} &< 2amp^m\\
&< 2ame^{47m^{3z+8}} \\
&< ae^{48m^{3z+8}} \\
&< ae^{e^{m\left(3 + \frac{6}{\log(m)}\right)}} 
\end{align*}

The last inequality can be checked with a computer for $m < 145$. For $m \ge 145$ we have $48m^8 < e^{\frac{3m}{2\log(m)}}$, which, combined with $m^{3z} < e^{3m + \frac{9m}{2\log(m)}}$ from Lemma~\ref{priemteller}, is sufficient. In conclusion we may say the following:

\begin{dirichlet}\label{diri}
If $r_i \neq 0$ for all $i$ with $\gcd(i, t) = 1$, then for all $a$ there exists a $b < ca$ for which $v_{a,b} < v_{a,b-1}$, where $c = e^{e^{m\left(3 + \frac{6}{\log(m)}\right)}}$.
\end{dirichlet}

\subsection{Bounding prime divisors} \label{primebound}
We could combine Theorem \ref{Theorem1} and Theorem \ref{ohyah} in Section \ref{dirichlet} when $\gcd(i,t) = 1$ implies $r_i \neq 0$, because in that case we always have $l < p$. However, in general this is not true. Consider for example $t = 2$, $r_1 = 0$, $r_2 = 1$. Then $p = 3$ divides $X_4 = 4(\frac{1}{2} + \frac{1}{4}) = 3$ and $n = l = 4 > 3 = p$. Luckily, we do not need $l < p$ to invoke Theorem \ref{Theorem1}; all we need is $\gcd(l, X_{a,b-1}) < p$. So we need to be able to bound prime divisors of either $l$ or $X_{a,b-1}$. In order to do this, recall $\mu_p = \left \lfloor \frac{\log(m-1)}{\log(p)} \right \rfloor$ and $e_p(r_i) \le \mu_p$ if $r_i \neq 0$.

\begin{pdoesntdivide}\label{Lemma3}
If $p \notin \Sigma_3$, then there exists a positive integer $c_p \le tp^{\mu_p}$ with $r_{c_p} \neq 0$, such that $e_p(X_n) \le \mu_p$ for all $k$ and $n$ with $c_pp^{\lambda k} \le n < (c_p + 1)p^{\lambda k}$.
\end{pdoesntdivide}

\begin{proof}
Fix $p \notin \Sigma_3$ for this proof and define $c_p$ to be the smallest integer $i$ for which the maximum of $e(i) - e(r_i)$ is attained, where $i$ runs from $1$ to $tp^{\mu_p}$. That is, $e(c_p) - e(r_{c_p}) \ge \displaystyle \max_{1 \le i \le tp^{\mu_p}}\big(e(i) - e(r_i)\big)$, with strict inequality for all $i < c_p$. \\

By Lemma \ref{notin3} an $i' \in \{p^{e(t) + \mu_p}, 2p^{e(t) + \mu_p}, \ldots, tp^{\mu_p}\}$ exists with $r_{i'} \neq 0$. We then get the lower bound $e(c_p) - e(r_{c_p}) \ge e(i') - e(r_{i'}) \ge \mu_p + e(t) - \mu_p = e(t)$. Moreover, this implies that $e(c_p) - e(r_{c_p})$ is non-negative, so $e(r_{c_p}) \neq \infty$ and $r_{c_p} \neq 0$. \\

Now we claim $e(X_{c_p}) \le e(L_{c_p}) + e(r_{c_p}) - e(c_p)$, because the only term $\frac{L_{c_p} r_i}{i}$ in the sum for $X_{c_p}$ which is not divisible by $p^{e(L_{c_p}) + 1 + e(r_{c_p}) - e(c_p)}$, is the term corresponding to $i = c_p$. Indeed, by the definition of $c_p$, for all $i < c_p$ we have $e(r_i) - e(i) \ge 1 + e(r_{c_p}) - e(c_p)$, implying $e\left(\frac{L_{c_p} r_i}{i}\right) > e\left(\frac{L_{c_p} r_{c_p}}{c_p}\right)$. \\

Let $k$ now be given and let $n$ be such that $c_pp^{\lambda k} \le n < (c_p + 1)p^{\lambda k}$. Analogously, we claim that only one term $\frac{L_{n} r_i}{i}$ does not vanish modulo $p^{e(L_{c_p}) + 1 + e(r_{c_p}) - e(c_p)}$, namely the term corresponding to $i = c_pp^{\lambda k}$. This would give us $e(X_n) \le e(L_{c_p}) + e(r_{c_p}) - e(c_p)$ as well. To prove this, we need a small lemma.

\begin{carmichael} \label{car}
If $p^{e(t)}$ divides $i$, then $r_i = r_{ip^{\lambda k}}$ for all $k \in \mathbb{N}$ .
\end{carmichael}

\begin{proof}
It is sufficient to prove $i \equiv ip^{\lambda k} \pmod{t}$. But this is equivalent to $(ip^{-e(t)}) \equiv (ip^{-e(t)})p^{\lambda k} \pmod{tp^{-e(t)}}$, which is true as $p^{\lambda k} \equiv 1 \pmod{tp^{-e(t)}}$ by the property of the Carmichael function that $d | t$ implies $\lambda(d) | \lambda(t)$. 
\end{proof}

When $i \le n$ is different from $c_pp^{\lambda k}$, we have $e\left(\frac{L_{n} r_i}{i}\right) = e(L_{c_p}) + \lambda k + e(r_i) - e(i)$. Now, by contradiction, if this is to be at most $e(L_{c_p}) + e(r_{c_p}) - e(c_p)$, then $e(i) - e(r_i) - \lambda k \ge e(c_p) - e(r_{c_p})$. The right-hand side of this inequality is at least $e(t)$, so if we define $j = ip^{-\lambda k} < c_p$, then $p^{e(t)} | j$. Since $e(j) = e(i) - \lambda k$ and $e(r_i) = e(r_j)$ by Lemma \ref{car}, we would have $e(j) - e(r_j) = e(i) - \lambda k - e(r_i) \ge e(c_p) - e(r_{c_p})$, which contradicts the definition of $c_p$. \\

So with $c_pp^{\lambda k} \le n < (c_p + 1)p^{\lambda k}$ we know $e(X_n) \le e(L_{c_p}) + e(r_{c_p}) - e(c_p)$. Now let $i \le c_p$ be such that $r_i \neq 0$ and $e(i) = e(L_{c_p})$. Then $e(X_n) \le e(L_{c_p}) + e(r_{c_p}) - e(c_p) \le e(i) + e(r_i) - e(i) = e(r_i) \le \mu_p$.
\end{proof}

Lemma \ref{Lemma3} should help us satisfy the condition $\gcd(l, X_{a,b-1}) < p$ from Theorem \ref{Theorem1}. However, for a prime divisor $q$ of $l$, even if $c_qq^{\lambda k} \le b-1 < (c_q + 1)q^{\lambda k}$ for some $k$, the astute reader might point out that we can only say something about $\gcd(l, X_{b-1})$ as opposed to $\gcd(l, X_{a,b-1})$. Fortunately, we have the following lemma.

\begin{xbisxab} \label{xbisxab}
If $c_qam \le c_qq^{\lambda k} \le b-1 < (c_q + 1)q^{\lambda k}$, then $e_q(X_{a,b-1}) = e_q(X_{b-1})$. 
\end{xbisxab}

\begin{proof}
First we note that $L_{a,b-1}$ is equal to $L_{b-1}$. Indeed, on the one hand we trivially have $L_{a,b-1} | L_{b-1}$. And as for the other direction, since $b-1 \ge am$, every integer $i$ smaller than $a$ with $r_{i} \neq 0$ has a multiple of the form $(jt+1)i$ between $a$ and $b-1$, with $r_{(jt+1)i} = r_{i} \neq 0$. So if $i$ divides $L_{b-1}$, it will also divide $L_{a,b-1}$, proving $L_{b-1} | L_{a,b-1}$ and therefore $L_{a,b-1} = L_{b-1}$. \\

Secondly, $L_{b-1}$ is divisible by $q^{\lambda k}$ since $r_{c_qq^{\lambda k}} = r_{c_q} \neq 0$, by Lemma \ref{Lemma3}. Therefore the only terms $\frac{L_{b-1} r_i}{i}$ in the sum for $X_{b-1}$ that are non-zero modulo $q^{e_q(X_{b-1}) + 1}$ are the ones where $i$ is divisible by $q^{\lambda k - e_q(X_{b-1})}$. The latter quantity is larger than $a$ as we assumed $q^{\lambda k} \ge am$, while $q^{e_q(X_{b-1})} \le q^{\mu_q} < m$, by Lemma \ref{Lemma3}. Since all terms that are non-zero modulo $q^{e_q(X_{b-1}) + 1}$ are larger than $a$, we indeed have $e_q(X_{a,b-1}) = e_q(X_{b-1})$. 
\end{proof}

\subsection{Diophantine approximation to the rescue}\label{dioph}
The proof of Theorem \ref{ohyah} is still valid for any integer $M \ge m$ instead of $m$, because the only property of $m$ that we used is that it is larger than $\max(r, t)$. In particular, with $M = \left\lfloor e^{2m + \frac{4m}{3\log(m)}} \right\rfloor$, let $n$ be the smallest integer $i \ge i_1$ for which a prime $p \ge M$ exists with $p | X_i$, and write $n = lp^k$ with $\gcd(l, p) = 1$. These integers $M$, $n$, $l$, $p$ and $k$ will now all be fixed for the rest of Section \ref{upper}. \\

If $l < p$, then the arguments from Section \ref{dirichlet} can be used again, and one can check $b(a) < a e^{e^{e^{4m}}}$. If the inequality $l < p$ does not hold however, we claim that we still have the weaker estimate $l < pt$. Indeed, by Lemma \ref{notin3} an $i \in \{1,2,\ldots,t\}$ exists with $r_{ip^{k+1}} \neq 0$. So if $l > pt \ge pi$, then $\frac{L_nr_n}{lp^k} \equiv 0 \pmod{p}$, contradicting the definition of $n$. We will therefore assume $p < l < pt$ from now on. It then turns out that $l$ must have a prime divisor $q$ for which $q^{e_q(l)}$ is large.


\begin{primeorpower} \label{primeorpower}
If for every prime divisor $q$ of $l$ we have $q < m$ and $q^{e_q(l)} < m^2$, then $l \le M$.
\end{primeorpower}

\begin{proof}
Let $l \in \mathbb{N}$ be such that for all prime divisors $q$ of $l$ we have $q < m$ and $q^{e_q(l)} < m^2$. Then $l \le \displaystyle \prod_{q < m} q^{\left\lfloor\frac{\log(m^2-1)}{\log(q)}\right\rfloor}$ and with a computer one can check that for $m < 2\cdot10^5$, this product is smaller than $e^{2m + \frac{4m}{3\log(m)}}$. For $m \ge 2\cdot10^5$, we will bound $l$ by using the inequalities $e^{0.98m} < \displaystyle \prod_{p \le m} p < e^{m(1 + \frac{1}{2\log(m)})}$ which can be found in \mbox{\cite{pnt2}} as Theorem $4$ and Theorem $10$, and the inequality $2\pi(m^{2/3}) < \frac{3.554m^{2/3}}{\log(m)}$ that follows from Lemma \ref{priemteller}.
\begin{align*}
l &\le \prod_{q < m} q^{\left\lfloor\frac{\log(m^2-1)}{\log(q)}\right\rfloor} \\
&= \prod_{q \le m^{2/3}} q^{\left\lfloor\frac{\log(m^2-1)}{\log(q)}\right\rfloor} \prod_{m^{2/3} < q < m} q^2 \\
&< m^{2\pi(m^{2/3})} \prod_{m^{2/3} < q < m} q^2 \\
&< e^{3.554m^{2/3}}  e^{-1.96m^{2/3}} e^{2m + \frac{m}{\log(m)}} \\
&= e^{1.594m^{2/3}} e^{2m + \frac{m}{\log(m)}} \\
&< e^{2m + \frac{4m}{3\log(m)}}
\end{align*}

Here the final inequality used $m \ge 2 \cdot 10^5$.
\end{proof}

So if $l > p \ge M$, then either $l$ is divisible by a prime $q \ge m$ or $l$ is divisible by a prime $q < m$ with $q^{e_q(l)} \ge m^2$. Let us fix this prime $q$ for the rest of Section \ref{upper}, and observe $q \notin \Sigma_3$ in either case. We therefore get by Lemma \ref{Lemma3} and Lemma \ref{xbisxab} that if $c_qam \le c_qq^{\lambda k_2} < b < (c_q + 1)q^{\lambda k_2}$, then $\gcd(l, X_{a,b-1}) \le lq^{\mu_q - e_q(l)} \le \frac{l}{m} < p$, which is the condition of Theorem \ref{Theorem1}. So we conclude the following:

\begin{notmuchleft}\label{close}
If $k_1$ and $k_2$ are positive integers such that with $b = np^{\lambda k_1}$ the string of inequalities $c_qam \le c_qq^{\lambda k_2} < b < (c_q + 1)q^{\lambda k_2}$ holds, then $v_{a,b} < v_{a,b-1}$.
\end{notmuchleft}

To find $k_1$ and $k_2$ for which these inequalities are satisfied, we have to do some Diophantine approximation.

\begin{dat}\label{dat}
There exist positive integers $b_1$ and $b_2$ with $b_2 < 2 \log(q) m^3$ such that the following inequality holds:

\begin{equation*}
\epsilon := |b_2 \log(p) - b_1 \log(q) | < \frac{1}{2m^3}
\end{equation*}
\end{dat}

\begin{proof}
Dirichlet's Approximation Theorem states that for any real number $\zeta > 0$ and any $N \in \mathbb{N}$, there exist positive integers $b_1$ and $b_2$ with $b_2 \le N$ such that $|b_2 \zeta - b_1| \le \frac{1}{N+1}$. Now we apply this with $\zeta = \frac{\log(p)}{\log(q)}$ and $N = \lfloor 2 \log(q) m^3 \rfloor$ to obtain $|b_2 \frac{\log(p)}{\log(q)} - b_1| < \frac{1}{2 \log(q) m^3}$. Multiplying both sides of the inequality by $\log(q)$ finishes the proof.
\end{proof}

\begin{kron} \label{kron}
Let $b_1, b_2$ and $\epsilon$ be as in Lemma \ref{dat}. Let $\gamma > 0$ be any positive real number and set $C = \left \lceil \frac{\gamma}{\epsilon} \right \rceil$. Then, if $b_2 \log(p) - b_1 \log(q) > 0$, we have

\begin{equation*}
0 \le Cb_2 \log(p) - Cb_1 \log(q) - \gamma < \frac{1}{2m^3}
\end{equation*}

while if $b_2 \log(p) - b_1 \log(q) < 0$, we have

\begin{equation*}
\frac{-1}{2m^3} < Cb_2 \log(p) - Cb_1 \log(q) + \gamma \le 0
\end{equation*}
\end{kron}

\begin{proof}
Assume $b_2 \log(p) - b_1 \log(q) > 0$. The other case can be proven in an analogous manner. Then, on the one hand:
\begin{align*}
Cb_2 \log(p) - Cb_1 \log(q) - \gamma &= C\big(b_2 \log(p) - b_1 \log(q)\big) - \gamma\\
&=  \left \lceil \frac{\gamma}{\epsilon} \right\rceil\epsilon - \gamma \\
&\ge \left(\frac{\gamma}{\epsilon}\right)\epsilon - \gamma \\
&= 0
\end{align*}

while on the other hand:
\begin{align*}
Cb_2 \log(p) - Cb_1 \log(q) - \gamma&= \left\lceil \frac{\gamma}{\epsilon} \right\rceil\epsilon - \gamma \\
&< \left(\frac{\gamma}{\epsilon} + 1 \right)\epsilon - \gamma\\
&= \epsilon \\
&< \frac{1}{2m^3} \qedhere
\end{align*}
\end{proof}

\begin{gammaslim}\label{gamm}
Let $D \in \mathbb{N}$ be any integer larger than or equal to $k+2$ and assume that we choose $\gamma$ in Lemma \ref{kron}, equal to

\begin{equation*}
\gamma  = \pm\left(\frac{\log(c_q) + \log(c_q + 1) - 2\log(n)}{2\lambda}\right) + D\log(p)
\end{equation*}

where plus or minus depends on whether $b_2 \log(p) - b_1 \log(q)$ is positive or negative, respectively. Then $\gamma > (D - k - 2)\log(p) \ge 0$ and the inequalities $c_qq^{\lambda k_2} < np^{\lambda k_1} < (c_q + 1)q^{\lambda k_2}$ hold, with $k_2 = Cb_1$ and $k_1 = Cb_2 \mp D$.
\end{gammaslim}

\begin{proof}
Let us first prove the lower bound on $\gamma$.
\begin{align*}
y &\ge D \log(p) - \left|\frac{\log(c_q) + \log(c_q + 1) - 2\log(n)}{2\lambda}\right| \\
&> D \log p - \max\big(\log(c_q + 1), \log(n)\big) \\
&\ge \min \big(D \log p - \log(m^2), D \log(p) - \log(lp^k)\big) \\
&> \min \big(D \log p - \log(p), D \log(p) - (k+2)\log(p)\big) \\
&= (D - k - 2)\log(p)
\end{align*}

To prove the inequalities $c_qq^{\lambda Cb_1} < np^{\lambda (Cb_2 \mp D)} < (c_q + 1)q^{\lambda Cb_1}$, we should consider two distinct cases, depending on whether $b_2 \log(p) - b_1 \log(q)$ is positive or negative. These proofs are however completely analogous to each other. So let us only do the first one and leave the second one as exercise for the reader. Assume $b_2 \log(p) - b_1 \log(q) > 0$ and let us first try to find an upper bound for $np^{\lambda k_1}$, taking Lemma \ref{kron} as a starting point.
\begin{align*}
Cb_2 \log(p) - Cb_1 \log(q) - \gamma &< \frac{1}{2m^3} \\
&< \frac{\log(c_q + 1) - \log(c_q)}{2\lambda} \\
\end{align*}

Here we used $\lambda < m$ and the fact $\log(x) - \log(x-1) > \frac{1}{x}$ with $x = c_q + 1 \le m^2$. Now we multiply by $\lambda$, apply the definition of $\gamma$, rearrange and exponentiate.
\begin{align*}
\lambda (Cb_2 - D) \log(p)  + \log(n) &< \lambda Cb_1 \log(q) + \log(c_q + 1) \\
np^{\lambda (Cb_2 - D)} &< (c_q + 1)q^{\lambda Cb_1}
\end{align*}

For a lower bound on $np^{\lambda k_1}$, we use similar ideas.
\begin{align*}
Cb_2 \log(p) - Cb_1 \log(q) - \gamma &\ge 0 \\
&> \frac{\log(c_q) - \log(c_q + 1)}{2\lambda}
\end{align*}

And once more we multiply by $\lambda$, use the definition of $\gamma$, rearrange and exponentiate.
\begin{align*}
\lambda (Cb_2 - D) \log(p)  + \log(n) &> \lambda Cb_1 \log(q) + \log(c_q) \\
np^{\lambda (Cb_2 - D )} &> c_qq^{\lambda Cb_1} \qedhere
\end{align*} 
\end{proof}

\begin{infinite}\label{inf}
For every $a$ there are infinitely many $b$ for which $v_{a,b} < v_{a,b-1}$.
\end{infinite}

\begin{proof}
The only inequality from Lemma \ref{close} that we have not checked yet is the inequality $c_qam \le c_qq^{\lambda k_2}$. Choose $D$ from Lemma \ref{gamm} to be any integer larger than $am + k + 2$. Then $c_qq^{\lambda k_2}$ is indeed larger than $c_qam$;
\begin{align*}
c_qq^{\lambda k_2} &= c_qq^{\lambda Cb_1} \\
&> c_qq^{\gamma} \\
&> c_qq^{am} \\
&> c_qam \qedhere
\end{align*}
\end{proof}

\subsection{Explicit bounds for all sequences} \label{final}
We are now in a position to prove our final theorem on upper bounds.

\begin{godsgift}\label{god}
For all $a$ there exists a $b < ca$ for which $v_{a,b} < v_{a,b-1}$, where $c = e^{e^{e^{m\left(4 + \frac{7}{\log(m)}\right)}}}$.
\end{godsgift}

\begin{proof}
Let us recall the chain of dependency. We chose $M = \left\lfloor e^{2m + \frac{4m}{3	\log(m)}} \right\rfloor$ to get $n = lp^k$ with $l > p \ge M$. Then a prime divisor $q$ of $l$ exists such that with $b = np^{\lambda k_1}$, we have $b < (c_q + 1)q^{\lambda Cb_1}$. With $Z = \pi(M-1)$ and $M > 10^5$, we now apply Lemma \ref{priemteller} to upper bound $n$, similar to what we did in Section \ref{dirichlet}. 
\begin{align*}
n &< 44M^{3Z+7} \\
&< e^{3.4M} \\
&< e^{e^{2m + \frac{1.8m}{\log(m)}}} \\
&< e^{e^{4m}} 
\end{align*}

Via similar reasoning we also get $pm < nm < e^{e^{2m + \frac{1.8m}{\log(m)}}} < e^{e^{4m}}$. \\

We can then upper bound $b$ by $(c_q + 1)q^{\lambda Cb_1} < m^2q^{m Cb_1} < (pm)^{2m Cb_1} < e^{2m Cb_1e^{4m}}$, where the first inequality follows from Lemma \ref{Lemma3} and the second inequality follows from $q \le l < pt < pm$ as explained at the start of Section \ref{dioph}. So we still need to find upper bounds for $C$ and $b_1$. \\

As for $b_1$, Lemma \ref{dat} gives us that it is smaller than $\frac{b_2 \log(p)}{\log(q)} + 1 < 2m^3\log(p) + 1 < 2m^3\log(pm) < 2m^3e^{4m}$. So $2mb_1 < 4m^4e^{4m} < e^{6m}$. And finally, we would like to find a bound for $C = \left \lceil \frac{\gamma}{\epsilon} \right \rceil < (\gamma + 1)\epsilon^{-1}$. We therefore need to bound both $\gamma$ and $\epsilon^{-1}$ and starting with $\epsilon^{-1}$, we use an effective version of Baker's Theorem on a lower bound on linear forms in logarithms.

\begin{baker} \label{baker}
Let $b_1, b_2$ and $\epsilon$ be as in Lemma \ref{dat}. Then we have the following lower bound:
\begin{equation*}
\log(\epsilon) = \log\left(|b_2 \log(p) - b_1 \log(q) |\right) > -e^{4m + \frac{6.9m}{\log(m)}}
\end{equation*}
\end{baker}

\begin{proof}	
We need to take a look at Corollary $2$ of \mbox{\cite[p. 288]{bak}} and the notation they use. In their notation, $\alpha_2$ equals our $p$, while $\alpha_1$ is our prime $q$. Furthermore, $b_1$ is our $b_1$ and $b_2 = b_2$. So $D$, which is defined in Section $2$ of that paper as $[\mathbb{Q}(\alpha_1, \alpha_2) : \mathbb{Q}]/[\mathbb{R}(\alpha_1, \alpha_2) : \mathbb{R}]$, simply equals $1$. We can let $\log(A_1)$ and $\log(A_2)$ be $\log(q)$ and $\log(p)$ respectively, which makes their $b' = \frac{b_1}{D \log(A_2)} + \frac{b_2}{D \log(A_1)}$ in our case bounded by $2\frac{b_2}{\log q} + 1 < 4m^3 + 1$, so that $\log(b') + 0.14 < \log(4m^3 + 1) + 0.14 < 5\log(m)$. And now we may apply Corollary $2$ of \mbox{\cite{bak}}.
\begin{align*}
\log(|b_2 \log(p) - b_1 \log(q)|) &\ge -24.34\left(\max \left\{5\log(m), 21, \frac{1}{2} \right\} \right)^2 \log(q) \log(p) \\
&> -5586 \log^2(m) \log^2(pm) \\
&> -5586 \log^2(m) e^{4m + \frac{3.6m}{\log(m)}} \\
&> -e^{4m + \frac{6.9m}{\log(m)}} \qedhere
\end{align*}
\end{proof}

To upper bound $\gamma = \gamma_D$, we use its definition as it was given in Lemma \ref{gamm}.
\begin{align*}
\gamma + 1 &= 1 \pm\left(\frac{\log(c_q) + \log(c_q + 1) - 2\log(n)}{2\lambda}\right) + D\log(p) \\
&< 1 + \max\bigl(\log(c_q + 1), \log(n)\bigr) + D\log(p) \\
&< 1 + \max\left(\log(m^2), e^{4m} \right) + De^{4m} \\
&= 1 + (D+1)e^{4m} \\
&< De^{5m}
\end{align*}

Here, by Lemma \ref{gamm} and the proof of Corollary \ref{inf}, we have to choose $D$ larger than or equal to $k+2$ and such that $q^{\lambda Cb_1} \ge am$, where $C$ depends on $\gamma$, which in turn depends on $D$. If $D = k+2$ already ensures $q^{\lambda Cb_1} \ge am$, then we choose $D = k+2$ and, by using $k < \lambda \le m-2$ (otherwise $p | X_{n'}$ with $n' = np^{-\lambda}$, contradicting the definition of $n$), the upper bound on $\gamma + 1$ simplifies to $\gamma + 1 < De^{5m} < me^{5m} < e^{6m}$. In this case we have:
\begin{align*}
b(a) &\le b \\
&< e^{2m Cb_1e^{4m}} \\
&< e^{e^{6m} e^{e^{4m + \frac{6.9m}{\log(m)}}} e^{6m} e^{4m}} \\
&< e^{e^{e^{4m + \frac{7m}{\log(m)}}}} \\
&= c \le ca
\end{align*}

In the other case we have to choose $D$ larger than $k+2$ to make sure $q^{\lambda Cb_1} \ge am$. So then we can choose $D$ in such a way that $q^{\lambda Cb_1} = q^{\lambda b_1 \left \lceil \gamma_D\epsilon^{-1} \right \rceil} \ge am > q^{\lambda b_1 \left \lceil \gamma_{D - 1}\epsilon^{-1} \right \rceil}$ holds, and we get: 
\begin{align*}
b(a) &\le b \\
&< (c_{q}+1)q^{\lambda b_1 \left \lceil \gamma_D\epsilon^{-1} \right \rceil} \\
& = (c_{q}+1)q^{\lambda b_1 \left \lceil \gamma_{D - 1}\epsilon^{-1} \right \rceil} q^{\lambda b_1 \left(\left \lceil \gamma_D\epsilon^{-1} \right \rceil - \left \lceil \gamma_{D - 1}\epsilon^{-1} \right \rceil\right)} \\
&< am^3 q^{\lambda b_1 \epsilon^{-1}  ((\gamma_D + 1)  - \gamma_{D - 1})} \\
&< a (pm)^{2\lambda b_1 \epsilon^{-1}  ((\gamma_D + 1)  - \gamma_{D - 1})} \\
&< a e^{2m b_1 \epsilon^{-1}  ((\gamma_D + 1)  - \gamma_{D - 1})e^{4m}} \\
&< a e^{e^{6m} e^{e^{4m + \frac{6.9m}{\log(m)}}}  e^{4m} e^{4m}} \\
&< a e^{e^{e^{4m + \frac{7m}{\log(m)}}}} \\
&= ca \qedhere
\end{align*}
\end{proof}

\newpage

\section{Lower bounds}\label{lower}
\subsection{A logarithmic lower bound} \label{lowerlog}
In the previous section we proved the upper bound $b(a) < ca$, for some constant $c$. Or, in other words, we can upper bound the difference $b(a) - a$ by a linear function. This difference turns out to grow at least logarithmically.

\begin{lima}\label{Theorem10}\label{limba}
We have the uniform lower bound ${\displaystyle \liminf_{a \rightarrow \infty}} \left(\frac{b(a)-a}{\log a}\right) \ge \frac{1}{2}$. 
\end{lima}

\begin{proof}
If $b$ is an integer with $a < b < a + \big(\frac{1}{2} - o(1)\big)\log(a)$, then we will prove that $b$ is not equal to $b(a)$. If $r_b = 0$, then we definitely have $b \neq b(a)$, so we may assume that $b$ is an integer with $r_b \neq 0$. We will then show $v_{a,b} > v_{a,b-1}$. \\

Recall $r = \max_i(|r_i|)$ and let, for this proof, $L_{b-a}$ and $L_r$ be the least common multiples of the elements in the sets $\{1, 2, \ldots, b-a\}$ and $\{1, 2, \ldots, r\}$ respectively, regardless of whether some $r_i$ are zero or not. On the other hand, $L_{a,b}$ is still the least common multiple of only those $i \in \{a, a+1, \ldots, b\}$ for which $r_i \neq 0$. \\

When $r_b \neq 0$, we have $L_{a,b} = \text{lcm}(L_{a,b-1}, b) = \frac{bL_{a,b-1}}{\gcd(L_{a,b-1}, b)}$. For a prime power divisor $p^k$ of $\gcd(L_{a,b-1}, b)$ we need $p^k \le b-a$, so that $\gcd(L_{a,b-1}, b) \le L_{b-a}$. We claim a similar upper bound on $g_{a,b}$, which will follow from the next lemma.

\begin{jamaarhoezodan} \label{noudaarom}
For all primes $p$ we have $e_p(g_{a,b}) \le e_p(L_{b-a}) + e_p(L_r)$.
\end{jamaarhoezodan}

\begin{proof}
We may assume $e(L_{a,b}) > e(L_{b-a}) + e(L_r)$, in which case there exists an $i \in [a, b]$ with $r_i \neq 0$ and $e(i) = e(L_{a,b}) > e(L_{b-a}) + e(L_r)$. Now, if $e(j)$ is also larger than $e(L_{b-a})$ for some $j \neq i$, then $|i - j| > b-a$ so that $j \notin [a, b]$. In other words, for all $j \in [a, b]$ with $j \neq i$ we have $e(j) \le e(L_{b-a}) < e(L_{a, b}) - e(L_r)$. This implies $\frac{L_{a,b}r_j}{j} \equiv 0 \pmod{p^{e(L_r)+1}}$ for all $j \in [a, b]$ different from $i$. We then obtain $X_{a, b} \equiv \frac{L_{a, b}r_i}{i} \pmod{p^{e(L_r)+1}}$, and we conclude $e(g_{a, b}) \le e(X_{a, b}) = e(r_i) \le e(L_r)$.
\end{proof}

Since $g_{a,b}$ is equal to the product of $p^{e_p(g_{a,b})}$ over all primes $p$, Lemma \ref{noudaarom} in particular implies $g_{a,b} \le L_{b-a}L_r$. To prove Theorem \ref{limba} we now apply the inequality $b > L^2_{b-a}L_r$, which follows from $b - a < \big(\frac{1}{2} - o(1)\big)\log(a)$ and PNT.
\begin{align*}
v_{a,b} &= \frac{L_{a,b}}{g_{a,b}} \\
&= \frac{bL_{a,b-1}}{\gcd(L_{a,b-1}, b)g_{a,b}} \\
&\ge \frac{bL_{a,b-1}}{L_{b-a}^2L_r} \\
&> L_{a,b-1} \\
&\ge v_{a,b-1} \qedhere
\end{align*}
\end{proof}

\subsection{Optimality of the lower bound} \label{lower0} 
As it turns out, the lower bound from the previous section is close to sharp, as we will show that the lower limit ${\displaystyle \liminf_{a \rightarrow \infty}} \left(\frac{b(a)-a}{\log a}\right)$ is finite for all sequences of $r_i$. 

\begin{generallowertight} \label{generaltight}
We have the upper bound ${\displaystyle \liminf_{a \rightarrow \infty}} \left(\frac{b(a)-a}{\log a}\right) \le t(t+1)\varphi(t)$. Moreover, if $t > 1$ and $r_i \neq 0$ for all $i$ with $\gcd(i,t) = 1$, then we can lower this bound to ${\displaystyle \liminf_{a \rightarrow \infty}} \left(\frac{b(a)-a}{\log a}\right) < 20 \log(\log(2t))$. And in the case where $r_i \neq 0$ holds for all $i$, this can be further improved to ${\displaystyle \liminf_{a \rightarrow \infty}} \left(\frac{b(a)-a}{\log a}\right) \le 2$. 
\end{generallowertight}

\begin{proof}
Let us recall what we did in Section \ref{iflarge}. There, $b$ was the product of two factors: a power of a prime $p > \max(r, t)$ for which $e_p(X_{a,b}) > e_p(X_{a,b-1}) = 0$, and a factor $l$, ideally with $l < p$ so that the inequality $\gcd(l, X_{a,b-1}) < p$ is immediate. To prove Theorem \ref{generaltight}, we will once again have a prime $p > \max(r, t)$ and then define $b = eQp^{\lambda k}$ as a product of three factors instead, where $k$ is large enough, $r_b \neq 0$, $e$ is smaller than $p$, and $Q$ is a certain product of distinct primes $q > p$, such that for every $q | Q$ there is exactly one $i \in [a, b-1]$ with $q | i$ and $r_i \neq 0$. \\

If we then choose $a = b - (e - c)p^{\lambda k}$ (where $e$ and $c < e$ will be defined shortly) with $r_a \neq 0$, we claim that the equality $L_{a,b} = L_{a,b-1}$ still holds, which is the analogue of Lemma \ref{lisl}. Indeed,  if $k$ is large enough, $L_{a,b-1}$ is divisible by $b - et$ and therefore by $e$, $L_{a,b-1}$ is divisible by $p^{\lambda k}$ because $a$ is, and $Q | L_{a,b-1}$ by the property that for every $q | Q$ there exists an $i \in [a, b-1]$ with $q | i$ and $r_i \neq 0$. If we assume (analogous to Lemma \ref{p-ness}) for the moment that $p$ divides $X_{a,b}$ but $p$ does not divide $X_{a,b-1}$, then we can copy our calculation of $g_{a,b}$ at the end of the proof of Theorem \ref{Theorem1} almost verbatim, but with $eQ$, instead of $l$. This results in $g_{a,b} \ge \frac{p}{\gcd(eQ, X_{a,b-1})} g_{a,b-1}$. Now we use the fact that for every prime $q | Q$ there is, by assumption, only one $i \in [a, b-1]$ with $q | i$ and $r_i \neq 0$, which makes $X_{a,b-1}$ congruent to $\frac{L_{a,b-1}r_i}{i} \not \equiv 0 \pmod{q}$. We therefore get $g_{a,b} > g_{a,b-1}$ by $e < p$, and we may still conclude $v_{a,b} < v_{a,b-1}$. \\

So we need to define $e$, $Q$ and $p$ such that the above properties hold, and ideally have $Q$ as large as possible, to ensure that $b-a$ is small compared to $a$.  \\

Let $1 \le i_1 < i_2 < \ldots$ be the sequence of indices $i$ for which $r_i \neq 0$, and define the quadratic polynomial $f(x) = x^2 + 2(i_2 - i_3)x + (i_3 - i_1)(i_3 - i_2)$. Now fix any prime $p$ larger than $2\max(r, t)$ such that $f(x) \equiv 0 \pmod{p}$ is solvable. One can check that $f(x)$ has a root modulo $p$ if, and only if, $(i_3 - i_2)(i_1 - i_2)$ is a quadratic residue modulo $p$. We will then separate the proof into two distinct cases, depending on the existence of a positive integer $j$ for which $r_{i_{j}}$ is different from $-r_{i_{j+1}}$. \\

Case I. A positive integer $j$ exists with $r_{i_{j}} \neq -r_{i_{j+1}}$. \\
To avoid too many double subscripts, define $c = i_{j}$ and $e = i_{j+1}$, and note that we may assume $e \le 2t$. We will introduce three products $Q_1, Q_2, Q_3$ of primes for which the congruence conditions $Q_i \equiv r_e(e - c)e^{-1}(r_e + r_c)^{-1} \pmod{p}$ and $Q_i \equiv 1 \pmod{t}$ hold. Note that we use the assumption $r_c \neq -r_e$ here, as otherwise $(r_e + r_c)^{-1}$ would not exist. The fact that $r_e(e - c)e^{-1}(r_e + r_c)^{-1} \pmod{p}$ is a non-zero residue class, follows from the assumption $p > 2\max(r, t)$. And for completeness' sake: if for some $i \in \{1, 2, 3 \}$ the definition that we will provide for $Q_i$ does not make sense, as no product exists for which the two congruence conditions are both true, then define $Q_i = 1$ instead.  \\

Let $Q_1$ be the largest product of the primes $q$ with $r_{e-q} \neq 0$ that are contained in the interval $\big(\frac{1}{2}(e-c)p^{\lambda k}, (e-c)p^{\lambda k}\big)$, such that the aforementioned congruence conditions on $Q_1$ hold. Let $Q_2$ be the largest product of the primes $q \in \big(\frac{1}{3}(e-c)p^{\lambda k}, \frac{1}{2}(e-c)p^{\lambda k}\big)$ with $r_{e-q} = 0 \neq r_{e-2q}$, for which the congruence conditions hold. And define $Q_3$ to be the largest product of the primes $q \in \big((\frac{e - c}{e - c + 1})p^{\lambda k}, p^{\lambda k}\big)$ with $q \equiv 1 \pmod{t}$, for which the congruence conditions hold. Finally, define $Q = \max(Q_1, Q_2, Q_3)$. \\

With $b = eQp^{\lambda k}$ and $a = b - (e - c)p^{\lambda k}$, we then claim that $v_{a,b}$ is smaller than $v_{a,b-1}$. As we mentioned at the start of this section, in order to prove this, we have to check $e_p(X_{a,b}) > e_p(X_{a,b-1}) = 0$, and we need to show that for every $q | Q$ there is exactly one $i \in [a, b-1]$ with $q | i$ and $r_i \neq 0$. Let us start with the latter. \\

This property is easiest seen for primes $q | Q_1$, since both $b - 2q$ and $b$ are then outside the interval $[a, b-1]$. We therefore see that $i = b-q$ is the only $i \in [a, b-1]$ with $q | i$, while $r_i = r_{b-q} = r_{e-q} \neq 0$ by the definition of $Q_1$. As for $q | Q_2$, the only multiples of $q$ that are contained in the interval $[a, b-1]$, are $b-q$ and $b-2q$. But $r_{b - q} = r_{e - q} = 0$, by the definition of $Q_2$. And so we see that $i = b - 2q$ is the only $i \in [a, b-1]$ with $q | i$ and $r_i = r_{e - 2q} \neq 0$. Finally, for a prime divisor $q$ of $Q_3$, the integers in the interval $[a, b-1]$ that are divisible by $q$ are precisely $b-q, b-2q, \ldots, b-(e-c)q$, since $b - (e-c+1)q < b - (e-c)p^{\lambda k} = a$. But $r_{b-iq} = r_{e-i} = 0$ for all $i$ with $1 \le i < e-c$, by the definitions of $c$ and $e$. This implies that $i = b - (e-c)q$ is the only $i \in [a, b-1]$ with $q | i$ and $r_i = r_c \neq 0$. Analogously, the only $i \in [a, b-1]$ with $p^{\lambda k} | i$ and $r_i \neq 0$ is $i = b - (e-c)p^{\lambda k} = a$. All in all we conclude $e_p(X_{a,b-1}) = e_q(X_{a,b-1}) = 0$ for all $q | Q$. \\

As for $X_{a,b} \pmod{p}$, there are now two integers $i \in [a, b]$ with $r_i \neq 0$ and $i$ divisible by $p^{\lambda k}$; $i = a$ and $i = b$. We therefore get the following:
\begin{align*}
X_{a,b} &= L_{a,b} \sum_{i=a}^b \frac{r_i}{i} & \\
&\equiv L_{a,b} \sum_{i=0}^{e-c} \frac{r_{b-ip^{\lambda k}}}{b - ip^{\lambda k}} &\pmod{p} \\
&\equiv \frac{L_{a,b}}{p^{\lambda k}} \sum_{i=0}^{e-c} \frac{r_{e-i}}{eQ - i} &\pmod{p} \\
&\equiv \frac{L_{a,b}}{p^{\lambda k}} \left(\frac{r_e}{eQ} + \frac{r_c}{eQ - (e - c)} \right) &\pmod{p} \\
&\equiv \left(\frac{L_{a,b}}{p^{\lambda k}eQ\big(eQ - (e - c)\big)}\right) \big(eQ(r_e + r_c) - r_e(e - c)\big) &\pmod{p} \\
&\equiv 0 &\pmod{p}
\end{align*}

The final equality follows from the congruence $Q \pmod{p}$ we imposed. The inequality $v_{a,b} < v_{a,b-1}$ now indeed follows from copying the calculation of $g_{a,b}$ at the end of Section \ref{iflarge}, but with $eQ$ instead of $l$. \\

Case II. For all $j \in \mathbb{N}$ we have $r_{i_{j}} = -r_{i_{j+1}}$. \\ 
In this case we define $c = i_1$, $d = i_2$ and $e = i_3$, and the congruence conditions for the $Q_i$ are now $f(eQ_i) \equiv 0 \pmod{p}$ and $Q_i \equiv 1 \pmod{t}$. For the definitions of $Q_1$ and $Q_2$ one can copy the definitions we used in Case I, the only distinction being the different congruence condition we have here. And $Q_3$ is now defined as the largest product of the primes $q \in \big(p^{\lambda k}, \big(\frac{e - c}{e - d}\big)p^{\lambda k}\big)$ with $q \equiv 1 \pmod{t}$, for which the congruence conditions hold. Once again, with $Q = \max(Q_1, Q_2, Q_3)$, $b = eQp^{\lambda k}$ and $a = b - (e - c)p^{\lambda k}$, we will show the inequality $v_{a,b} < v_{a,b-1}$ in an analogous manner. \\

Since the definitions of $Q_1$ and $Q_2$ are still the same as they were in the previous case, the proofs that for every prime $q | Q_1Q_2$ there is only one $i \in [a, b-1]$ with $q|r_i$ and $r_i \neq 0$, are still the same as well. As for $q | Q_3$, the integers in the interval $[a, b-1]$ that are divisible by $q$ are $b-q, b-2q, \ldots, b - \left\lfloor\frac{b-a}{q}\right\rfloor q$. The term $\left\lfloor\frac{b-a}{q}\right\rfloor = \left\lfloor\frac{(e-c)p^{\lambda k}}{q}\right\rfloor$ is at least $e-d$ (since $q < \big(\frac{e-c}{e-d}\big)p^{\lambda k}$) and smaller than $e-c$ (since $q > p^{\lambda k}$). Since, by the definitions of $c, d$ and $e$, $r_{b - iq} = r_{e - i} = 0$ for all $i \neq e-d$ with $1 \le i < e-c$, we once again deduce that there is only one $i \in [a, b-1]$ (namely $i = b - (e-d)q$) with $q | i$ and $r_i  \neq 0$. The analogous calculation for $X_{a,b-1} \pmod{p}$ contains two non-zero terms in this case, i.e. the two terms corresponding to $i = b - (e-d)p^{\lambda k}$ and $i = b - (e-c)p^{\lambda k} = a$;
\begin{align*}
X_{a,b-1} &= L_{a,b-1} \sum_{i=a}^{b-1} \frac{r_i}{i} &\\
&\equiv L_{a,b-1} \sum_{i = 1}^{e-c} \frac{r_{b-ip^{\lambda k}}}{b - ip^{\lambda k}} &\pmod{p} \\
&\equiv \frac{L_{a,b-1}}{p^{\lambda k}} \sum_{i = 0}^{e-c} \frac{r_{e-i}}{eQ - i} &\pmod{p} \\
&\equiv \frac{L_{a,b-1}}{p^{\lambda k}} \left(\frac{r_d}{eQ - (e - d)} + \frac{r_c}{eQ - (e - c)} \right) &\pmod{p} \\
&\equiv \frac{L_{a,b-1}r_d}{p^{\lambda k}} \left(\frac{1}{eQ - e + d} - \frac{1}{eQ - e + c} \right) &\pmod{p} \\
&\equiv \frac{L_{a,b-1}r_d(c-d)}{p^{\lambda k}(eQ - e + d)(eQ - e + c)} &\pmod{p} \\
&\not \equiv 0 &\pmod{p}
\end{align*}

On the other hand, the sum $X_{a,b} \pmod{p}$ also contains the term corresponding to $i = b$;
\begin{align*}
X_{a,b} &\equiv \frac{L_{a,b}r_e}{p^{\lambda k}} \left(\frac{1}{eQ} - \frac{1}{eQ - e + d} + \frac{1}{eQ - e + c} \right) &\pmod{p} \\
&\equiv \left(\frac{L_{a,b}r_e}{p^{\lambda k}eQ(eQ - e + d)(eQ - e + c)}\right) f(eQ) &\pmod{p} \\
&\equiv 0 &\pmod{p}
\end{align*}

Where the final equality follows from the congruence $Q \pmod{p}$ we imposed for this case. And the conclusion $v_{a,b} < v_{a,b-1}$ once again follows. What remains to be done is calculate (a lower bound on) the size of $Q$, which will give us an upper bound on $\frac{b - a}{\log(a)}$. \\

In both Case I and Case II it follows from PNT that the product $Q_3$ (together with the congruence conditions) exists if $k$ is large enough. In Case $I$ we have $a > Q \ge Q_3 \ge \exp\left[\frac{(1 + o(1))p^{\lambda k}}{(t+1)\varphi(t)}\right]$, which implies the upper bound $b - a \le tp^{\lambda k} \le \big(t(t+1)\varphi(t) + o(1)\big) \log(a)$. As for Case II, note that the sequence $r_1, r_2, \ldots, r_t$ must contain at least two non-zero terms, as otherwise all non-zero terms would be equal to one another, contradicting the assumption of Case II. We therefore deduce $e - c \le t$ and $\frac{e-c}{e-d} \ge \frac{t}{t-1}$. This gives us $a > Q \ge Q_3 \ge \exp\left[\frac{(1 + o(1))p^{\lambda k}}{(t-1)\varphi(t)}\right]$ by PNT, implying $b - a \le tp^{\lambda k} \le \big(t(t-1)\varphi(t) + o(1)\big) \log(a)$. In either case we are done and this finishes the proof for arbitrary sequences. \\

For non-zero sequences we note that $Q_1$ is divisible by all primes in the interval $\big(\frac{1}{2}(e-c)p^{\lambda k}, (e-c)p^{\lambda k}\big)$, so that $a > Q \ge Q_1 = \exp\big[\big(\frac{1}{2} + o(1)\big)(e - c)p^{\lambda k}\big]$ and $b - a = (e-c)p^{\lambda k} \le \big(2 + o(1)\big)\log(a)$. We may therefore assume from now on that $r_i$ is non-zero for all $i$ coprime to $t$, with $t > 1$. \\

Define $S_1$ to be the set of positive integers $i < t$ coprime to $t$ with $r_{e - i} \neq 0$ and define $S_2$ to be the set of positive integers $i < t$ coprime to $t$ with $r_{e - i} = 0 \neq r_{e - 2i}$. From PNT it follows that $Q_1 = \exp\big[\frac{(1 + o(1)) |S_1| (e-c)  p^{\lambda k}}{2\varphi(t)}\big]$ and $Q_2 = \exp\big[\frac{(1 + o(1)) |S_2| (e-c) p^{\lambda k}}{6\varphi(t)}\big]$, which then gives us the upper bound $b - a = (e-c)p^{\lambda k} \le \big(1 + o(1)\big) \min\big(\frac{2\varphi(t)}{|S_1|}, \frac{6\varphi(t)}{|S_2|}\big) \log(a)$. It therefore suffices to show $\min\big(\frac{2\varphi(t)}{|S_1|}, \frac{6\varphi(t)}{|S_2|}\big) < 20 \log(\log(2t))$. \\

When $t = 2, 4, 6$, one can check that either $S_1$ or $S_2$ is non-empty, so that $\min\big(\frac{2\varphi(t)}{|S_1|}, \frac{6\varphi(t)}{|S_2|}\big) \le 6\varphi(t) < 20 \log(\log(2t))$. For $t = 3,5$ we have $|S_1| \ge 1$, which implies $\min\big(\frac{2\varphi(t)}{|S_1|}, \frac{6\varphi(t)}{|S_2|}\big) \le 2\varphi(t) < 20 \log(\log(2t))$. We may therefore assume $t \ge 7$ from now on. We will then use the following lemma, where $s_2 = 1$ and $s_q = 2$ for $q > 2$.

\begin{gregmartin} \label{greg}
The union $S_1 \cup S_2$ has at least $t \displaystyle \prod_{q | t} \left(1 - \frac{s_q}{q}\right)$ elements.  
\end{gregmartin}

\begin{proof}
For a positive integer $i < t$, define $i' = i$ if both $e$ and $t$ are even, and define $i' = 2i$ otherwise. Furthermore note that $i \in S_1 \cup S_2$, if $\gcd(i, t) = \gcd(e - i', t) = 1$. The goal is to count how many such $i$ there are, and we will first do this if $t$ is a prime power. \\

If $t$ is a power of $2$, then $\gcd(i, t) = \gcd(e - i', t) = 1$ for all odd $i < t$. On the other hand, if $t$ is a power of an odd prime $q$, then $\gcd(i, t) = \gcd(e - i', t) = 1$ for all $i < t$, unless $i \equiv 0 \pmod{q}$ or $e \equiv i' \pmod{q}$. The result for general $t$ now follows from the Chinese Remainder Theorem.\footnote{We thank Greg Martin for (the inspiration for) this proof, see \mbox{\cite{mo2}}.} \qedhere 
\end{proof}

In order to apply Lemma \ref{greg}, we need to be able to lower bound the product that occurs in its statement.

\begin{prodoverprimes} \label{primeprod}
For any set $S$ of odd primes $q$ we have the following inequality:

\begin{equation*}
\prod_{q \in S} \left(1 - \frac{2}{q}\right) > 0.62 \prod_{q \in S} \left(1 - \frac{1}{q}\right)^{2}
\end{equation*}
\end{prodoverprimes}

\begin{proof}
For $q = 3$ we have $\left(1 - \frac{2}{q}\right) = \frac{3}{4}\left(1 - \frac{1}{q}\right)^2$. For $q \ge 5$ we apply the inequality $\left(1 - \frac{2}{q}\right) > \left(1 - \frac{1}{q^2}\right)^{2} \left(1 - \frac{1}{q}\right)^{2}$, which can be checked by expanding the brackets. We now deduce our result from the equality $\prod_{q \ge 5} \left(1 - \frac{1}{q^2}\right)^2 = \frac{81}{\pi^4}$, which in turn follows from the Euler product for the Riemann zeta function.
\end{proof}

From Lemma \ref{greg} it follows that $\max(3|S_1|, |S_2|)$ is at least $\frac{3}{4}t \prod_{q | t} \left(1 - \frac{s_q}{q}\right)$, and we can combine this inequality with Lemma \ref{primeprod}.
\begin{align*}
\min\left(\frac{2\varphi(t)}{|S_1|}, \frac{6\varphi(t)}{|S_2|}\right) &= \frac{6\varphi(t)}{\max(3|S_1|, |S_2|)} \\
&\le \frac{8t\prod_{q| t} \left(1 - \frac{1}{q}\right)}{t\prod_{q| t} \left(1 - \frac{s_q}{q}\right)} \\
&< \frac{13\prod_{q | t \text{ odd }} \left(1 - \frac{1}{q}\right)}{\prod_{q | t \text{ odd }} \left(1 - \frac{1}{q}\right)^2} \\
&= 13\prod_{q | t \text{ odd }} \left(1 - \frac{1}{q}\right)^{-1} 
\end{align*}

This latter quantity is equal to $\frac{6.5t'}{\varphi(t')}$ where $t' = t$ or $t' = 2t$, depending on whether $t$ is even or odd. We need to manually verify that this is smaller than $20 \log(\log(2t))$ for $7 \le t \le 40$. For $t \ge 41$ we apply inequality $(3.42)$ from \mbox{\cite[p. 72]{pnt2}}:
\begin{align*}
\frac{6.5t'}{\varphi(t')} &< 6.5e^{\gamma} \log(\log(2t)) + \frac{6.5 \cdot 2.51}{\log(\log(t))} \\
&\le 6.5e^{\gamma} \log(\log(2t)) + \frac{6.5 \cdot 2.51}{\log(\log(41))} \\
&< 11.58 \log(\log(2t)) + 8.42 \log(\log(82)) \\
&\le 20 \log(\log(2t)) \qedhere
\end{align*}
\end{proof}

\subsection{Improvements in the classical case} \label{lower2}
When $r_1 = t = 1$, we can further strengthen Theorem \ref{generaltight}.

\begin{limc}\label{limco}
If $r_i = 1$ for all $i$, then $0.54 < {\displaystyle \liminf_{a \rightarrow \infty}} \left(\frac{b(a)-a}{\log a}\right) < 0.61$.
\end{limc}

In order to show these tighter bounds on the lower limit, divisibility properties of the polynomials $f_d(x) := \displaystyle \sum_{i=0}^d \prod_{\substack{j=0 \\ j\neq i}}^d (x-j)$ turn out to be important. We therefore define $\delta(f_d)$ to be the density of primes $p$ such that $f_d(x) \equiv 0 \pmod{p}$ is solvable. By a (slight extension of a) theorem of Frobenius which we will meet shortly (see Lemma \ref{frobb}), this density exists and one can in principle calculate it. With $c$ defined as $\displaystyle \sum_{d = 1}^{\infty} \frac{\delta(f_d)}{d(d+1)}$, the proof of Theorem \ref{limco} is a combination of the following three lemmas.

\begin{upperlinc} \label{upperlink}
If $r_i = 1$ for all $i$, then ${\displaystyle\liminf_{a \rightarrow \infty}} \left(\frac{b(a)-a}{\log a}\right) \le \frac{1}{2c}$.
\end{upperlinc}

\begin{lowerlinc} \label{lowerlink}
If $r_i = 1$ for all $i$, then ${\displaystyle\liminf_{a \rightarrow \infty}} \left(\frac{b(a)-a}{\log a}\right) \ge \frac{1}{1+c}$.
\end{lowerlinc}

\begin{cbounds} \label{cbounds}
We have the inequalities $0.82 < c < 0.85$, from which $\frac{1}{2c} < 0.61$ and $\frac{1}{1+c} > 0.54$ follow by computation. 
\end{cbounds}

To prove Lemma \ref{upperlink} and Lemma \ref{lowerlink}, we have to introduce some more notation. Let, with $n$ a large integer and $1 \le d \le \sqrt{n} - 1$, $S_d$ be the set of primes $p$ with $\frac{n}{d+1} < p \le \frac{n}{d}$ such that $f_d(x) \equiv 0 \pmod{p}$ is solvable, and let $x_p$ be any root of $f_d(x) \pmod{p}$. Conversely, $T_d$ denotes the set of primes $p$ with $\frac{n}{d+1} < p \le \frac{n}{d}$ for which $f_d(x) \equiv 0 \pmod{p}$ is not solvable. We furthermore define $S$ and $T$ as the union of the sets $S_d$ and $T_d$ respectively, over all $d$ with $1 \le d \le \sqrt{n} - 1$. Moreover, $Q$ and $P$ are defined as the products of all primes $p \in S$ and $p \in T$ respectively, and for a prime divisor $p$ of $Q$, let us define $Q_p = \frac{Q}{p}$.  \\

From the existence of $\delta(f_d)$, it follows by PNT that $\frac{|S_d|}{\pi(n)}$ converges for fixed $d$ to $\frac{\delta(f_d)}{d(d+1)}$. We therefore get $Q = e^{(c + o(1))n}$ and $P = e^{(1 - c + o(1))n}$. And when $p \in S_d$, we have the following lemma for the roots $x_p$ of $f_d(x) \pmod{p}$.

\begin{xpsolution} \label{xpsol}
For all $i$ with $0 \le i \le d$ we have $x_p \not\equiv i \pmod{p}$.
\end{xpsolution}
 
\begin{proof}
By contradiction; assume $x_p \equiv i \pmod{p}$ for some $i$ with $0 \le i \le d$. Then $0 \equiv f_d(x_p) \equiv \displaystyle \prod_{\substack{j=0 \\ j \neq i}}^d (x_p-j) \pmod{p}$ and by Euclid's lemma $x_p - j \equiv 0 \pmod{p}$ for some $j \neq i$. This gives $i \equiv j \pmod{p}$, which is impossible as $0 < |i - j| \le d < \frac{n}{d+1} < p$. 
\end{proof}

We can now prove Lemma \ref{upperlink}.

\begin{proof}[Proof of Lemma \ref{upperlink}]
Let $q$ be the largest prime in $S_2$, so that we have $f_2(x_q) = 3x_q^2 - 6x_q + 2 \equiv 0 \pmod{q}$. Then $x'_q = -x_q +2$ is a root of $f_2(x) \pmod{q}$ as well, since $f_2(x'_q) = 3(-x_q + 2)^2 - 6(-x_q + 2) + 2 = 3x_q^2 - 6x_q + 2 \equiv 0 \pmod{q}$. Moreover $x'_q = -x_q + 2 \not\equiv x_q \pmod{q}$ as otherwise $x_q \equiv 1 \pmod{q}$, which contradicts Lemma \ref{xpsol}. So $x_q$ and $x'_q$ are two distinct roots of $f_2(x) \pmod{q}$. \\


Let $x_0$ and $x_1$ be the unique positive integers smaller than $Q$ such that the following congruences hold: $x_0 \equiv x_1 \equiv x_pQ_p^{-1} \pmod{p}$ for all $p \in S \setminus \{q\}$, $x_0 \equiv x_qQ_q^{-1} \pmod{q}$ and $x_1 \equiv x'_qQ_q^{-1} \pmod{q}$. Then $x_0$ and $x_1$ differ by a multiple of $Q_q$ as they are congruent modulo every prime divisor of $Q_q$, so at least one of them is larger than $Q_q$. Define $x = \max(x_0, x_1) > Q_q$ and redefine $x_q := x'_q$ if $x_1 > x_0$, so that $x \equiv x_pQ_p^{-1} \pmod{p}$ holds for all $p \in S$. \\

With $a$ and $b$ defined as $b = xQ$ and $a = b - n$ respectively, we claim $v_{a,b} < v_{a,b-1}$. Since $a = \big(1 - o(1)\big)b$ and $b = xQ > \frac{Q^2}{q} = e^{(2c + o(1))n}$, this would finish the proof of Lemma \ref{upperlink}. To prove that $v_{a,b}$ is indeed smaller than $v_{a,b-1}$, we need some results on the prime divisors of $g_{a,b}$ and $g_{a,b-1}$. 


\begin{xpnietklein} \label{psquaredfree}
For all $p \in S$, $L_{a,b}$ is not divisible by $p^2$.
\end{xpnietklein}

\begin{proof}
The integers in $[a, b]$ that are divisible by $p \in S_d$ are $b$, $b-p, \ldots, b-dp$, as $b - dp \ge b - n > b - (d+1)p$. Since $\frac{b - ip}{p} = xQ_p - i \equiv x_p - i \not\equiv 0 \pmod{p}$ for $0 \le i \le d$ by Lemma \ref{xpsol}, we see that $b - ip$ is not divisible by $p^2$ for any $0 \le i \le d$, so $L_{a,b}$ is not divisible by $p^2$ either.
\end{proof}

\begin{pxabopnieuw} \label{thisbefore}
For all $p \in S$, $X_{a,b}$ is divisible by $p$, while $p$ does not divide $X_{a,b-1}$.
\end{pxabopnieuw}

\begin{proof}
This should be reminiscent of Lemma \ref{p-ness}. For a prime divisor $p$ of $Q$ with $p \in S_d$, let us calculate $X_{a,b} \pmod{p}$.
\begin{align*}
X_{a,b} &= L_{a,b} \sum_{i=a}^b \frac{1}{i} & \\
&\equiv L_{a,b} \sum_{i=0}^d \frac{1}{b - ip} &\pmod{p} \\
&\equiv \frac{L_{a,b}}{p} \sum_{i=0}^d \frac{1}{xQ_p - i} &\pmod{p} \\
&\equiv \frac{L_{a,b}}{p} \sum_{i=0}^d \frac{1}{x_p - i} &\pmod{p} \\
&\equiv \frac{L_{a,b}}{p} \frac{f_d(x_p)}{\prod_{i=0}^d (x_p - i)} &\pmod{p} \\
&\equiv 0 &\pmod{p}
\end{align*}

On the other hand, $p$ does not divide $\frac{L_{a,b}}{b}$ by Lemma \ref{psquaredfree}. This implies $0 \equiv X_{a,b} = \frac{L_{a,b}}{L_{a,b-1}}X_{a,b-1} + \frac{L_{a,b}}{b} \not \equiv \frac{L_{a,b}}{L_{a,b-1}}X_{a,b-1} \pmod{p}$, from which we conclude that $X_{a,b-1}$ is not divisible by $p$.
\end{proof}

And now we can finish the proof of Lemma \ref{upperlink}. For all primes $p \in S$, we have $e_p(L_{a,b}) = e_p(L_{a,b-1}) = 1$ by Lemma \ref{psquaredfree}, which implies $e_p(g_{a,b}) = 1$ and $e_p(g_{a,b-1}) = 0$ by Lemma \ref{thisbefore}. On the other hand, for all primes $p \notin S$, we have $e_p(g_{a,b-1}) \le e_p(g_{a,b}) + \min\big(e_p(L_n), e_p(x)\big)$ by Lemma \ref{Lemma1} and Lemma \ref{noudaarom}. Adding this inequality to the equality $e_p(L_{a,b}) = e_p(L_{a,b-1}) + \max\big(0, e_p(x) - e_p(L_n)\big)$ gives $e_p(L_{a,b}) + e_p(g_{a,b-1}) \le e_p(L_{a,b-1}) + e_p(g_{a,b}) + e_p(x)$ for all $p \notin S$. Combining both the estimates on the primes that do and do not belong to S, and we get:
\begin{align*}
L_{a,b}g_{a,b-1} &= \prod_{p \text{ prime }} p^{e_p(L_{a,b}) + e_p(g_{a,b-1})} \\
&= \prod_{p \in S} p^{e_p(L_{a,b}) + e_p(g_{a,b-1})} \prod_{p \notin S} p^{e_p(L_{a,b}) + e_p(g_{a,b-1})} \\
&\le \prod_{p \in S} p^{e_p(L_{a,b-1}) + e_p(g_{a,b}) - 1} \prod_{p \notin S} p^{e_p(L_{a,b-1}) + e_p(g_{a,b}) + e_p(x)} \\
&= \prod_{p \text{ prime }} p^{e_p(L_{a,b-1}) + e_p(g_{a,b})} \prod_{p \in S} p^{-1} \prod_{p \notin S} p^{e_p(x)} \\
&= L_{a,b-1}g_{a,b} \frac{x}{Q} \\
&< L_{a,b-1}g_{a,b} \\
\end{align*}

We therefore have $v_{a,b} = \frac{L_{a,b}}{g_{a,b}} < \frac{L_{a,b-1}}{g_{a,b-1}} = v_{a,b-1}$. 
\end{proof}

To prove Lemma \ref{lowerlink}, we need need one other lemma.

\begin{tnogood} \label{notgood}
If $b-a = n$, then $e_p(g_{a,b}) \le e_p(g_{a,b-1})$ for all $p \in T$.
\end{tnogood}

\begin{proof}
If $e_p(L_{a,b}) \ge 2$ for some prime $p \in T_d$, then $p^2$ does not divide any other integer in the interval $[a,b]$, since $b - a = n < p^2$. So we see $X_{a,b} \not\equiv 0 \pmod{p}$, as we only have one non-zero term modulo $p$. We are therefore free to assume $e_p(L_{a,b}) = 1$. If $p$ does not divide $b$, then $X_{a,b} = \frac{L_{a,b}}{L_{a,b-1}}X_{a,b-1} + \frac{L_{a,b}}{b} \equiv \frac{L_{a,b}}{L_{a,b-1}}X_{a,b-1} \pmod{p}$ which is equal to zero if and only if $X_{a,b-1} \equiv 0 \pmod{p}$ as well. On the other hand, if $p$ does divide $b$, then we can follow the analogous calculation of $X_{a,b} \pmod{p}$ in Lemma \ref{thisbefore}. This implies $X_{a,b} \not\equiv 0 \pmod{p}$, as otherwise $f_d(x) \equiv 0 \pmod{p}$ would be solvable, contrary to $p \in T_d$. 
\end{proof}

\begin{proof}[Proof of Lemma \ref{lowerlink}]
If $b-a = n < \frac{\log(a)}{1 + c + o(1)} < \frac{\log(b)}{1 + c + o(1)}$, then $b > e^{(1 + c + o(1))n}$. Now, by combining Lemma \ref{notgood} with Lemma \ref{noudaarom}, we get the inequality $\frac{g_{a,b}}{g_{a,b-1}} \le \frac{L_n}{P}$. A calculation similar to the one at the end of the proof of Theorem \ref{limba} then implies $v_{a,b} \ge \frac{L_{a,b-1}}{g_{a,b-1}} \frac{bP}{L_n^2} = v_{a,b-1} \frac{bP}{L_n^2}$. And with $P = e^{(1 - c + o(1))n}$, $L_n = e^{(1 + o(1))n}$ and $b > e^{(1 + c + o(1))n}$, we may finally deduce $v_{a,b} > v_{a,b-1}$.
\end{proof}

To prepare the proof of Lemma \ref{cbounds}, we need information on the values of $\delta(f_d)$, in order to be able to estimate $c$. As it turns out, $\delta(f_d) = 1$ for all odd $d$.

\begin{polproperties} \label{polprop}
For all $d \in \mathbb{N}$ and all $x \in \mathbb{R}$ we have $f_d(x) = (-1)^d f_d(d-x)$. In other words, $f_d(x+\frac{d}{2})$ is an odd function when $d$ is odd and it is an even function when $d$ is even. In particular, $\delta(f_d) = 1$ when $d$ is odd.\footnote{This was suggested by Will Jagy, see \mbox{\cite{mo}}.}
\end{polproperties}

\begin{proof}
By direct calculation:
\begin{align*}
f_d(d-x) &= \sum_{i=0}^d \prod_{\substack{j=0 \\ j\neq i}}^d \big((d-x)-j\big) \\
&= \sum_{i=0}^d (-1)^d \prod_{\substack{j=0 \\ j\neq i}}^d \big(x - (d-j)\big) \\
&= (-1)^d \sum_{i=0}^d \prod_{\substack{j=0 \\ j\neq i}}^d (x - j) \\
&= (-1)^d f_d(x)
\end{align*}

Plugging in $\frac{d}{2} + x$ gives $f_d(\frac{d}{2} + x) = (-1)^d f_d(\frac{d}{2} - x)$, which implies $f_d(\frac{d}{2}) = 0$ when $d$ is odd. This in turn implies $f_d(x) \equiv 0 \pmod{p}$ with $x \equiv 2^{-1}d \pmod{p}$, for all odd primes $p$.
\end{proof}

Now, for a group of permutations on a set $X$, we say a permutation $\sigma$ is a derangement if $\sigma(x) \neq x$ for all $x \in X$. To find the value of $\delta(f_d)$ for even $d$, the number of derangements in the Galois group of $f_d(x)$ will be important.

\begin{frob} \label{frobb}
Let $G_d$ be the Galois group of $f_d(x)$, viewed as a group of permutations on the set of roots of $f_d(x)$. If $f_d(x)$ is irreducible, the density $\delta(f_d)$ is equal to the proportion of $\sigma \in G_d$ such that $\sigma$ is not a derangement.
\end{frob}

\begin{proof}
See \mbox{\cite{frob}} for a nice survey with references. They generally work with monic polynomials there, but this assumption can be omitted.
\end{proof}

Define $S_l^{+}$ to be the signed symmetric or hyperoctahedral group, which is the group of permutations $\sigma$ on $\{-l, -l+1, \ldots, -1, 1, 2, \ldots, l\}$ such that $\sigma(i) = -\sigma(-i)$, for all $i$. We then have the following result:

\begin{even} \label{evenn}
When $d = 2l$ is even, $G_d$ is isomorphic to a subgroup of $S_l^{+}$.
\end{even}

\begin{proof}
Define $g_d(x) = f_d(x + \frac{d}{2})$. By Lemma \ref{polprop}, $g_d(x)$ is even and this makes it slightly easier to work with. As $g_d(x)$ and $f_d(x)$ are translates of each other, they have the same Galois group, so it suffices to find the Galois group of $g_d(x)$. Let $\{x_{-l}, x_{-l-1}, \ldots, x_{-1}, x_1, \ldots, x_l\}$ be the roots of $g_d(x)$ with $x_i = -x_{-i}$ and let $\sigma$ be an element of $G_d$. If $\sigma(x_i) = x_j$, then $\sigma(-x_i) = -x_j$, since $\sigma$ is a field automorphism. We can thusly define an injective homomorphism $\phi$ from $G_d$ to $S_l^{+}$ such that for all $i$, if $\sigma \in G_d$ sends $x_i$ to $x_j$, then $\phi(\sigma)$ sends $i$ to $j$.
\end{proof}

Whenever $G_d$ is isomorphic to the full group $S_l^{+}$, we have an exact formula for the number of elements that are not derangements.

\begin{derange} \label{derange}
The fraction of elements in $S_l^{+}$ that are not derangements is equal to $1 - \displaystyle \sum_{i=0}^l \frac{(-1)^i}{2^ii!}$. 
\end{derange}

\begin{proof}
This follows directly from Theorem 2.1 in \mbox{\cite[p. 3]{derange}}, by applying the fact that $S_l^{+}$ contains $2^ll!$ integers.
\end{proof}

\begin{proof}[Proof of Lemma \ref{cbounds}]
Using the functions \textit{polisirreducible} and \textit{GaloisGroup} from the computer programs PARI/GP and Magma respectively, we have found that $f_d$ is irreducible for all even $d \le 500$, while $G_d$ is isomorphic to $S_l^{+}$ for all even $d \le 60$, except for $d = 8, 24, 48$. We can then apply Lemma \ref{derange} in order to find lower and upper bounds on $c$.
\begin{align*}
c &= \displaystyle \sum_{d = 1}^{\infty} \frac{\delta(f_d)}{d(d+1)} \\ 
&= \sum_{l = 1}^{\infty} \frac{\delta(f_{2l-1})}{2l(2l-1)} + \sum_{\substack{1 \le l \le 30 \text{ and} \\ l \notin \{4,12,24\}}} \frac{\delta(f_{2l})}{2l(2l+1)} + \sum_{\substack{l \ge 31 \text{ or} \\ l \in \{4,12,24\}}} \frac{\delta(f_{2l})}{2l(2l+1)}\\
&= \sum_{l = 1}^{\infty} \frac{1}{2l(2l-1)} + \sum_{\substack{1 \le l \le 30 \text{ and} \\ l \notin \{4,12,24\}}} \frac{1 - \displaystyle \sum_{i=0}^l \frac{(-1)^i}{2^ii!}}{2l(2l+1)} + \sum_{\substack{l \ge 31 \text{ or} \\ l \in \{4,12,24\}}} \frac{\delta(f_{2l})}{2l(2l+1)} 
\end{align*}

The first sum equals $\log(2) \approx 0.6931$ and the second sum is approximately equal to $0.1281$, giving $c > 0.82$. On the other hand, applying $\delta(f_{2l}) \le 1$ gives $0.025$ as an upper bound for the third sum, so that $c < 0.85$.
\end{proof}






\newpage

\section{Generalizations}\label{noperiod}
\subsection{Perfect powers as denominators} \label{powerino}
Let $d$ be a positive integer. It seems natural to look at sums of the form $\displaystyle \sum_{i=a}^b \frac{r_i}{i^d}$ to see which results, if any, still hold in this more general case. We will focus on the results from Section \ref{upper} and, for a start, it is possible to generalize Theorem \ref{Theorem1} with essentially the same proof. We will use analogous definitions ($L_{a,b}$ should now be the least common multiple of all integers $i^d \in \{a^d, (a+1)^d, .., b^d \}$ for which $r_i \neq 0$) and to specify the dependence on $d$, $b_d(a)$ will denote the smallest $b$ such that $v_{a,b} < v_{a,b-1}$. Let, analogous to Section \ref{iflarge}, $p \ge m = 1 + \max(r,t)$ be a prime number such that $p | X_{lp^k}$, where $lp^k \ge i_1$ is the smallest such integer. Let $k_1$ be an integer with $p^{\lambda k_1 + k} \ge \max(a, 2t)$ and choose $b = lp^{\lambda k_1 + k}$. We then obtain the following generalization of Theorem \ref{Theorem1}.

\begin{Theoremonevtwo} \label{onetwo}
If $\gcd(l^d, X_{a,b-1}) < p$, then $v_{a,b} < v_{a,b-1}$. Furthermore, if the condition $\gcd(l, X_{a,b-1}) < p$ is satisfied for the smallest $k_1$ such that $p^{\lambda k_1 + k} \ge \max(a, 2t)$ holds, then $b_d(a) \le \max(a-1,2t-1)lp^{\lambda}$.  
\end{Theoremonevtwo}

The only difference here is $l^d$ instead of $l$, in the condition $\gcd(l^d, X_{a,b-1}) < p$. And this condition is of course harder to satisfy when $d$ is large. For a prime divisor $q$ of $l$, recall that Lemma \ref{Lemma3} provided intervals $I$ such that $e_q(X_n)$ is small for all $n \in I$. Now, it is possible to generalize Lemma \ref{Lemma3} so that it works for general $d$. Unfortunately, this is not sufficient to guarantee that $\gcd(l^d, X_{a,b-1}) < p$ holds, due to potential other prime divisors of $l$. \\

One way to try to get around this problem is to search for positive integers $n$ such that, simultaneously for all prime divisors $q_i$ of $l$, $e_{q_i}(X_n)$ is bounded. If one assumes that the terms $\theta_i = \frac{\log(q_1)}{\log(q_i)}$ are rationally independent, then this can be done along the same lines as the proposed proof of Theorem $4$ in \mbox{\cite[p. 5]{dhn}}. And as we mentioned in Section \ref{intro}, the rational independence of the $\theta_i$ does follow from Schanuel's Conjecture, but is currently unknown. But we do get the following corollary:

\begin{finitedsjonnie} \label{sjonnie}
If Schanuel's conjecture is true, then $b_d(a)$ is finite, for all positive integers $d$ and $a$.
\end{finitedsjonnie}

Another idea to ensure that the inequality $\gcd(l^d, X_{a,b-1}) < p$ holds for some $b$, is to try to make sure that $l = q^k$ is itself a prime power. And somewhat surprisingly, here a large value of $d$ can actually be advantageous. 

\begin{easyasabc} \label{abc}
Let $i$ and $j > i$ be the smallest two (positive) indices such that $r_i$ and $r_j$ are non-zero. There exists an absolute constant $K$ such that for all $M \ge m$ and all $d > Ke^{M(1 + \frac{3}{\log(M)})}$, $X_j$ is divisible by a prime $p > M$.
\end{easyasabc}

\begin{proof}
One can check that $X_j$ is equal to $\frac{r_ij^d + r_ji^d}{\gcd(i, j)^d}$. By defining $g = \gcd(r_ij^d, r_ji^d)$, $A = g^{-1}r_{i}j^d$ and $B = g^{-1}r_{j}i^d$, we then get that $A$ and $B$ are coprime, and $A+B$ divides $X_j$. We will prove that $A+B$ has a large prime divisor, by applying known bounds on the abc conjecture. But first we have to show a lower bound on $A+B$ itself.

\begin{largedlargesum} \label{largedlargesum}
If $d > 2m\log(2m)$, then $|A+B| > e^{\frac{d}{2}}$.
\end{largedlargesum}

\begin{proof}
We first provide a lower bound on the ratio $|A/B|$, by using the inequalitiy $\log(1+x) > \frac{x}{2}$, which is valid for all $x$ with $0 < x \le 1$.
\begin{align*}
|A/B| &> \frac{1}{m}\left(\frac{j}{i}\right)^d \\
&> 2e^{-\log(2m)}\left(\frac{m+1}{m}\right)^d \\
&= 2e^{d \log(1 + \frac{1}{m}) - \log(2m)} \\
&> 2e^{\frac{d}{2m} -\log(2m)} \\
&> 2
\end{align*}

On the other hand, $|A| \ge \frac{1}{m-1} \left(\frac{j}{\gcd(i,j)}\right)^d \ge \frac{2^d}{m-1}$, since $\frac{j}{\gcd(i,j)} \in \mathbb{N}$ and $j > i \ge \gcd(i,j)$. Combining these bounds, we get the following:
\begin{align*}
|A + B| &\ge |A| - |B| \\
&> |A| - \frac{1}{2}|A| \\
&> \frac{2^d}{2(m-1)} \\
&> \frac{e^{\frac{2d}{3}}}{e^{\frac{d}{6}}} \\
&= e^{\frac{d}{2}} \qedhere
\end{align*}
\end{proof}

Let $rad(x)$ be the radical of $x$; the largest squarefree divisor of $x$. We then have the following lower bound on $rad(A + B)$.

\begin{radineqs} \label{radineq}
There exists an absolute constant $K \ge 1$ such that $rad(A+B) > \frac{2\log(|A+B|)}{Km^4}$.
\end{radineqs}

\begin{proof}
Since $\displaystyle \max_{x \ge 1} \textstyle \frac{\log^3(x)}{x^{\frac{2}{3}}} < 5$, Theorem $1$ from \mbox{\cite[p. 170]{abc}} implies (for some constant $c$) $|A+B| < \exp\big(5c \cdot rad(AB)rad(A+B)\big)$. Equivalently, we get that $rad(A+B)$ is larger than $\frac{\log(|A+B|)}{5c \cdot rad(AB)}$. The lemma now follows by proving $rad(AB) < 2m^4$ and taking $K = \max(1, 20c)$.
%
\begin{align*}
rad(AB) &\le rad(r_ij^dr_ji^d) \\
&\le rad(r_i)rad(j^d)rad(r_j)rad(i^d) \\
&= rad(r_i)rad(j)rad(r_j)rad(i) \\
&< 2m^4
\end{align*}

Here, the final inequality follows from $\max(|r_i|, |r_j|, i) < m$ and $j < 2m$.
\end{proof}

By combining Lemma \ref{largedlargesum} and Lemma \ref{radineq}, we get $rad(A+B) > \frac{d}{Km^4}$. In particular, if $d > Ke^{M(1 + \frac{3}{\log(M)})} > Km^4e^{M(1 + \frac{1}{2\log(M)})}$, then $rad(A + B) > e^{M(1 + \frac{1}{2\log(M)})}$, which, by Theorem 4 from \mbox{\cite[p. 70]{pnt2}}, is larger than the product of all primes smaller than or equal to $M$. We therefore conclude that $A+B$ must be divisible by a prime larger than $M$. 
\end{proof}

\begin{stormer} \label{storm}
If at least two out of $r_1, r_2, r_3, r_4, r_5$ are non-zero, then for all but finitely many $d$, $b_d(a)$ is finite for all $a$.
\end{stormer}

\begin{proof}[Proof (sketch).]
If at least two out of $r_1, r_2, r_3, r_4, r_5$ are non-zero, then $j = q^k$ from Lemma \ref{abc} is a prime power. We can then choose $M = m$ to get a prime divisor $p > m$ of $X_j$, while a generalization of Lemma \ref{Lemma3} provides intervals $I$ such that $q^{e_q(X_n)} < m < p$, for all $n \in I$. The arguments from Section \ref{dioph} can be generalized to work for general $d$ as well, and then provide infinitely $b$ for which $v_{a,b} < v_{a,b-1}$. Finally, Baker's method (see Section \ref{final}) allows one to make everything explicit again.
\end{proof}

\subsection{Perfect powers in the classical case}
With the notation of the previous section, we will now consider the case where $r_i = 1$ for all $i$. Let $p_d$ be the smallest prime $p$ for which $p-1$ does not divide $d$, set $j = \frac{1}{2}(p_d - 1)$, define $q_i$ to be the smallest prime divisor of $X_i$, and let $c_d$ be the smallest constant such that $b_d(a) \le c_d\max(1,a-1)$ holds for all $a \in \mathbb{N}$. Recall that Corollary $\ref{classical}$ gave us $c_1 = 6$, since $b_1(1) = b_1(2) = 6$. It is possible to generalize this and calculate $c_d$ for all $d$.

\begin{powers} \label{powerr}
If $d$ is odd, then $c_d = 6$. For even $d$ we have the (in)equalities $c_d = b_d(1) = \displaystyle \min_{2 \le i \le j} (iq_i) \le \textstyle jp_d$.
\end{powers}

\begin{proof}
Let $b$ be equal to $b_d(1)$ for this proof. Since $v_{1,b} < v_{1,b-1}$, we see that $g_b := \gcd(X_b, L_b)$ is larger than $g_{b-1}$. With $p$ any prime divisor for which $e_p(g_b) > e_p(g_{b-1})$, we claim that $p$ divides $b$. First, $b$ is not a power of $p$, as otherwise $X_b \equiv \frac{L_b}{b^d} \not \equiv 0 \pmod{p}$. But if $p$ does not divide $b$, then $X_{b} = X_{b-1} + \frac{L_b}{b^d} \equiv X_{b-1} \pmod{p^{e_p(L_b)}}$, contradicting $e_p(g_b) > e_p(g_{b-1})$. Now with $b = lp$, we see that $p$ must divide $X_l$, in much the same way as the proof of Lemma \ref{p-ness}. And since $p | X_l$, we conclude $c_d \ge b_d(1) = lp \ge \displaystyle \min_{i \ge 2} (iq_i)$. \\

On the other hand, we claim $\gcd(i^d, X_{a,n}) = 1 < q_i$ for all $a$, $n \ge a$, and $i$ with $2 \le i < p_d$. With $l = i$ and $p = q_i$, the upper bound on $b_d(a)$ in Theorem \ref{onetwo} then simplifies and can be rewritten as $c_d \le iq_i$. And this upper bound holds for all $i$ with $2 \le i < p_d$.

\begin{whendismultiple} \label{dismultiple}
Let $p$ be a prime such that $p-1$ divides $d$. Then $p$ does not divide $X_{a,n}$, for all positive integers $a$ and $n \ge a$. In particular, $X_{a,n}$ does not have any prime divisors smaller than $p_d$, and $\gcd(i^d, X_{a,n}) = 1$ for all $i < p_d$.
\end{whendismultiple}

\begin{proof}
Assume that $p^{dk}$ exactly divides $L_{a,n}$ and let $j_1$ and $j_2$ be such that $(j_1-1)p^k < a \le j_1p^k \le j_2p^k \le n < (j_2+1)p^k$ with $1 \le j_1\le j_2 \le p-1$. Then let us take a look at $X_n \pmod{p}$, and use the fact that $d$ is a multiple of $\varphi(p) = p-1$, which implies $i^d \equiv 1 \pmod{p}$ for all $i$ with $1 \le i \le p-1$.
\begin{align*}
X_{a,n} &\equiv \frac{L_{a,n}}{p^{dk}} \sum_{i=j_1}^{j_2} \frac{1}{i^d} &\pmod{p} \\
&\equiv \frac{L_{a,n}}{p^{dk}} (j_2+1 - j_1) &\pmod{p} 
\end{align*}

And this is non-zero since $1\le j_2 + 1 - j_1 \le p-1$.
\end{proof}

To recap, we now have $\displaystyle \min_{i \ge 2} (iq_i) \le b_d(1) \le c_d \le \displaystyle \min_{2 \le i < p_d} (iq_i)$, with $q_i \ge p_d$ for all $i$. \\

For odd $d$ we have $p_d = q_2 = 3$, so this string of inequalities becomes a string of equalities, and $c_d = 6$. When $d$ is even we will show $q_j = p_d$, which implies ${\displaystyle\min_{i \ge 2}} (iq_i) = {\displaystyle \min_{2 \le i \le j}} (iq_i) \le jp_d$, finishing the proof of Theorem \ref{powerr}.

\begin{pdivpowersum} \label{pdivpower}
Let $p$ be a prime such that $p-1$ does not divide $d$. If $d$ is even, then $X_{\frac{1}{2}(p-1)} \equiv 0 \pmod{p}$.
\end{pdivpowersum}

\begin{proof}
Let $g$ be a primitive root of $p$ and recall that $\{g, 2g, \ldots, (p-1)g\}$ and $\{\frac{1}{1}, \frac{1}{2}, \ldots, \frac{1}{p-1} \}$ are both complete sets of non-zero residues modulo $p$. In particular we see $\displaystyle \sum_{i=1}^{p-1} (ig)^d \equiv \sum_{i=1}^{p-1} i^d  \equiv \sum_{i=1}^{p-1} \frac{1}{i^d} \pmod{p}$, and we use this to prove that $p$ divides $X_{\frac{1}{2}(p-1)}$.
\begin{align*}
0 &\equiv L_{\frac{1}{2}(p-1)}  \sum_{i=1}^{p-1} \left((ig)^d - i^d\right) &\pmod{p} \\
&\equiv (g^d - 1) L_{\frac{1}{2}(p-1)}  \sum_{i=1}^{p-1} i^d &\pmod{p} \\
&\equiv (g^d - 1) L_{\frac{1}{2}(p-1)} \sum_{i=1}^{p-1} \frac{1}{i^d} &\pmod{p} \\
&\equiv (g^d - 1)L_{\frac{1}{2}(p-1)} \left(\sum_{i=1}^{\frac{1}{2}(p-1)} \frac{1}{i^d} + \sum_{i=1}^{\frac{1}{2}(p-1)} \frac{1}{(-i)^d} \right) &\pmod{p}\\
&\equiv 2(g^d - 1)X_{\frac{1}{2}(p-1)} &\pmod{p}
\end{align*}

Since $p-1$ does not divide $d$, we know $p \neq 2$. Moreover, $p$ does not divide $g^d - 1$ either, as $g$ is a primitive root of $p$ and $p-1 \nmid d$. We therefore conclude that $X_{\frac{1}{2}(p-1)}$ must be divisible by $p$. 
\end{proof}
\vspace{-0.55cm}\phantom\qedhere
\end{proof}

\begin{linnik}
For all $d$, $c_d = O\big(\log^{10}(d)\big)$. On the other hand, there are infinitely many $d$ with $c_d > 3\log(d)$.
\end{linnik}

\begin{proof}
Let $c$ be a small enough constant and $q$ be a prime smaller than $cp_d^{\frac{1}{5}}$. Then in \mbox{\cite{lin2}} it is proven that there exists a prime $p < p_d$ such that $p \equiv 1 \pmod{q}$. Since $d$ is divisible by $p-1$ for all $p < p_d$, $q$ divides $d$ as well. Therefore $d \ge \displaystyle \prod_{q < cp_d^{\frac{1}{5}}} q = e^{(1 + o(1))cp_d^{\frac{1}{5}}}$, implying $p_d = O\big(\log^5(d)\big)$. Since $c_d < p_d^2$, the upper bound follows. For the lower bound, choose $d = \text{lcm}(1, 2, 4, 6, 10, \ldots, p_d - 1)$, and note $d \le 2L_{\frac{1}{2}(p_d - 1)}$. Since $L_n < e^{1.04n}$ by Theorem $12$ in \mbox{\cite[p. 71]{pnt2}}, we get $d < 2e^{\frac{1.04}{2}(p_d - 1)} < e^{\frac{2}{3}p_d}$ and $c_d = \displaystyle \min_{i \ge 2}(iq_i) \ge 2p_d > 3\log(d)$.
\end{proof}

\begin{somecdvalues} \label{somecds}
$c_d = \begin{cases} 

6 &\mbox{if } d \equiv 1 \pmod{2} \\ 
10 &\mbox{if } d \equiv 2 \pmod{4} \\  
21 &\mbox{if } d \equiv 4,8 \pmod{12} \\ 
34 &\mbox{if } d \equiv 12 \pmod{24} \\ 
55 &\mbox{if } d \equiv 24,48,72,96 \pmod{120} \\ 
\end{cases}$
\end{somecdvalues}

\begin{proof}
All of these can be relatively quickly checked by calculating $p_d$, finding the possible values of $q_i$ for the first few $i$, and applying $c_d = \min_i (iq_i) \le jp_d$, when $d$ is even. Let us do this for the final case of $d \equiv 24,48,72,96 \pmod{120}$, and leave the rest for the interested reader. So we will assume that $24$ divides $d$ but $5$ does not divide $d$. Since $24$ is divisible by $1, 2, 4$ and $6$, but not by $10$, we see $p_d = 11$ and, using Theorem \ref{powerr}, we obtain $c_d \le 55$ right away. Furthermore, we claim that $X_i$ is not divisible by $13$ for any $i$, not divisible by $17$ for $i \le 3$ and not divisible by $19$ or $23$ for $i = 2$, so that $iq_i$ is minimized for $i = 5$, $q_i = 11$. To prove that $X_i$ is not divisible by $13, 17, 19$ or $23$ for the relevant values of $i$, let us deal with them one prime at a time. \\

By Lemma \ref{dismultiple} we have that $13$ does not divide $X_i$ for any $i$, as $12 | 24$. We furthermore have $\frac{1}{i^d} \equiv \pm 1 \pmod{17}$, as $8 | d$. But $\frac{1}{2^8} \equiv 1 \pmod{17}$, so that $X_2 \equiv 2 \pmod{17}$, while $X_3 \pmod{17}$ is either $1$ or $3$, and definitely non-zero as well. Finally, the only way either $19$ or $23$ divides $X_2$ is if $\frac{1}{2^d}$ is congruent to $-1$ modulo $19$ or $23$. But for $23$ this congruence is not solvable, while $2$ is a primitive root modulo $19$, so $\frac{1}{2^d} \equiv -1 \pmod{19}$ precisely when $d \equiv 9 \pmod{18}$. But this is impossible as $d$ is even.
\end{proof}

With the help of a computer it is not hard to extend Corollary \ref{somecds}. For example, $17 | X_6$ when $d \equiv 120 \pmod{240}$, $37 | X_3$ when $d \equiv \pm 240 \pmod{720}$, $p_d = 23$ when $\gcd(d, 11 \cdot 720) = 720$, $p_d = 29$ when $\gcd(d, 7 \cdot 7920) = 7920$ and $193$ divides $X_2$ when $d = 7\cdot7920$. Working this all out gives $c_d \le 406$ for $d < 110880$. \\

Theorem \ref{powerr} shows that $c_d$ is always equal to $b_d(1)$. But analogously to Theorem \ref{ogplus}, the upper bound on $b_d(a)$ can often be improved upon, for larger values of $a$. Let $C_d$ be the smallest constant such that $b_d(a) \le C_d(a-1)$ holds for all $a \ge 4$.\footnote{We choose $a \ge 4$ just because it happens to work in all cases we will consider. We conjecturally have $b_d(a) < (1+\epsilon)a$ for large enough $a$.} Then for all even $d < 120$ we can improve Corollary \ref{somecds}.

\begin{somebettercdvalues} \label{bettercd}
$C_d \le \begin{cases} 

\frac{25}{3} = 8.\overline{3} &\mbox{if } d \equiv 2 \pmod{4}  \\  
\frac{147}{8} = 18.375 &\mbox{if } d \equiv 4,8 \pmod{12}  \\ 
\frac{34}{3} = 11.\overline{3} &\mbox{if } d \equiv 12 \pmod{24}  \\ 
\frac{55}{3} = 18.\overline{3} &\mbox{if } d = 24  \\ 
\frac{111}{5} = 22.2 &\mbox{if } d = 48  \\ 
\frac{1587}{47} \approx 33.8 &\mbox{if } d = 72  \\ 
\frac{605}{23} \approx 26.3 &\mbox{if } d = 96  \\ 
\end{cases}$
\end{somebettercdvalues}

\begin{proof}[Proof (sketch)]
We will not give all the details, but instead construct functions $f_d(a)$ such that the motivated reader can check themselves that $v_{a,f_d(a)} < v_{a,f_d(a)-1}$ and $f_d(a) \le C_d(a-1)$ hold whenever $f_d(a)$ is defined, using the ideas that were already present in Section \ref{return}. Moreover, in every case we make sure that if $f_d(a) = lp^k$ (where the meaning of $p$ in the different cases should be clear), then every prime divisor $q$ of $l$ will be such that $q-1$ divides $d$, so that $\gcd(l^d, X_{a,f_d(a)-1}) = 1 < p$ follows immediately from Lemma \ref{dismultiple} and does not have to be checked separately. Finally, there is little doubt that these values can be extended and improved upon even further, but this paper is long enough as it is.

$$
\hspace{6pt}\text{If } d \equiv 2 \hspace{-5pt} \pmod{4}: f_{d}(a) = \begin{cases} 

10 &\mbox{if } 3 \le a \le 5  \\ 
21 &\mbox{if } a = 6 \text { and } d \equiv 2,10 \pmod{12} \\ 
26 &\mbox{if } a = 6 \text { and } d \equiv 6 \pmod{12} \\
9\cdot5^{k-1} &\mbox{if } 5^k < a \le 6\cdot5^{k-1} \mbox{ for some } k \ge 2 \\ 
2\cdot5^{k+1} &\mbox{if } 6\cdot5^{k-1} < a \le 5^{k+1} \mbox{ for some } k \ge 1 \\ 

\end{cases} 
$$

$$
\hspace{-10pt}\text{If } d \equiv 4,8 \pmod{12}: f_{d}(a) = \begin{cases} 

21 &\mbox{if } 3 \le a \le 7  \\ 
78 &\mbox{if } a = 8 \\ 
10\cdot7^{k-1} &\mbox{if } 7^k < a \le 8\cdot7^{k-1} \mbox{ for some } k \ge 2 \\ 
3\cdot7^{k+1} &\mbox{if } 8\cdot7^{k-1} < a \le 7^{k+1} \mbox{ for some } k \ge 1\\ 

\end{cases} 
$$

$$
\text{If } d \equiv 12 \hspace{-5pt} \pmod{24}: f_{d}(a) = \begin{cases} 

7\cdot17^{k} &\mbox{if } 17^k < a \le 2\cdot17^{k} \mbox{ for some } k \ge 1 \\ 
8\cdot17^{k} &\mbox{if } 2\cdot17^k < a \le 3\cdot17^{k} \mbox{ for some } k \ge 1 \\ 
2\cdot17^{k+1} &\mbox{if } 3\cdot17^k < a \le 17^{k+1} \mbox{ for some } k \ge 0 \\ 

\end{cases} 
$$

$$
\hspace{45pt}\text{If } d = 24: f_{d}(a) = \begin{cases} 

8\cdot11^{k} &\mbox{if } 11^k < a \le 2\cdot11^{k} \mbox{ for some } k \ge 1 \\ 
9\cdot11^{k} &\mbox{if } 2\cdot11^k < a \le 3\cdot11^{k} \mbox{ for some } k \ge 1 \\ 
5\cdot11^{k+1} &\mbox{if } 3\cdot11^k < a \le 11^{k+1} \mbox{ for some } k \ge 0 \\ 

\end{cases} 
$$

$$
\hspace{45pt}\text{If } d = 48: f_{d}(a) = \begin{cases} 

55 &\mbox{if } 4 \le a \le 5  \\ 
16\cdot37^{k} &\mbox{if } 37^{k} < a \le 2\cdot37^{k} \mbox{ for some } k \ge 1 \\ 
17\cdot37^{k} &\mbox{if } 2\cdot37^{k} < a \le 3\cdot37^{k} \mbox{ for some } k \ge 1 \\ 
18\cdot37^{k} &\mbox{if } 3\cdot37^{k} < a \le 4\cdot37^{k} \mbox{ for some } k \ge 1 \\ 
34\cdot37^{k} &\mbox{if } 4\cdot37^{k} < a \le 5\cdot37^{k} \mbox{ for some } k \ge 1 \\ 
3\cdot37^{k+1} &\mbox{if } 5\cdot37^{k} < a \le 37^{k+1} \mbox{ for some } k \ge 0 \\ 

\end{cases} 
$$

$$
\hspace{45pt}\text{If } d = 72: f_{d}(a) = \begin{cases} 

69 &\mbox{if } 4 \le a \le 23  \\ 
68 &\mbox{if } a = 47 \\ 
9\cdot23^{k} &\mbox{if } 23^k < a \le 2\cdot23^{k} \mbox{ for some } k \ge 1	\\ 
49\cdot23^{k-1} &\mbox{if } 2\cdot23^k < a \le 47\cdot23^{k-1} \mbox{ for some } k \ge 2	\\ 
3\cdot23^{k+1} &\mbox{if } 47\cdot23^{k-1} < a \le 23^{k+1} \mbox{ for some } k \ge 1 \\ 

\end{cases} 
$$

$$
\hspace{45pt}\text{If } d = 96: f_{d}(a) = \begin{cases} 

55 &\mbox{if } 4 \le a \le 11  \\ 
111 &\mbox{if } a = 23 \\ 
7\cdot11^{k} &\mbox{if } 11^k < a \le 2\cdot11^{k} \mbox{ for some } k \ge 1 \\ 
27\cdot11^{k-1} &\mbox{if } 2\cdot11^k < a \le 23\cdot11^{k-1} \mbox{ for some } k \ge 2 \\ 
5\cdot11^{k+1} &\mbox{if } 23\cdot11^{k-1} < a \le 11^{k+1} \mbox{ for some } k \ge 1 \\
\end{cases} 
$$  
\end{proof}




\subsection{Non-periodic sequences of numerators}
In this section we will drop the periodicity assumption on the sequence of $r_i$, and merely assume that there exists an $m$ such that $|r_i| < m$ for all $i$. We then ask ourselves: which, if any, of our results generalize to this case? For example, can we still prove upper or lower bounds on $b(a)$? \\

As it turns out, for upper bounds the answer is no. Perhaps somewhat surprisingly, given almost any set of integers $A$, if all we assume is that $r_i \in A$ for all $i$, then we cannot even exclude the possibility that $v_{1,n} = L_n$ holds for all $n \in \mathbb{N}$, unless $A$ is of a special form. More precisely:

\begin{erniesquestion}
If $A$ is a set of integers which contains at least one odd integer, and, for every odd prime $p$, there exist $a_1, a_2 \in A$ such that $a_1 \not \equiv a_2 \pmod {p}$, then it is possible to assign the $r_i$ values in $A$, such that the denominator of $\displaystyle \sum_{i=1}^n \frac{r_i}{i}$ equals $L_n$ for all $n \in \mathbb{N}$.
\end{erniesquestion}

\begin{proof}
We will prove this via induction. For a start, it does not matter what the value of $r_1$ is. Assume now that we have chosen $r_1, r_2, \ldots, r_{n-1} \in A$ so that $\displaystyle \frac{X_{n-1}}{L_{n-1}} = \sum_{i=1}^{n-1} \frac{r_i}{i}$ with $\gcd(X_{n-1}, L_{n-1}) = 1$. Then we will show that we can choose $r_n \in A$ so that $\gcd(X_n, L_n) = 1$ holds as well. \\

In general, $\gcd(X_n, L_n) = 1$ is equivalent to the statement that the smallest prime divisor of $X_n$ is larger than $n$. In particular, with the induction hypothesis we assume $X_{n-1} \not \equiv 0 \pmod{q}$ for all primes $q \le n-1$. Now there are three different cases to consider. \\

Case I. The integer $n$ is a prime power. \\
Assume $n = p^k$, let $q \neq p$ be any other prime smaller than $n$ and choose an arbitrary $r_n \in A$ that is not divisible by $p$. We then claim that both $p$ and $q$ do not divide $X_n$, by applying the fact that this case is the only one where $L_n \neq L_{n-1}$ and, more precisely, $L_n = pL_{n-1}$. On the one hand, $X_n = pX_{n-1} + \frac{L_n r_n}{n} \equiv \frac{L_n r_n}{n} \not \equiv 0 \pmod{p}$. While on the other hand, $X_n = pX_{n-1} + \frac{L_n r_n}{n} \equiv pX_{n-1} \not \equiv 0 \pmod{q}$, by the induction hypothesis. \\

Case II. One can write $n = lp^k$, for some $1 < l < p$ and $k \ge 1$. \\
In this case we claim that this prime $p$ is unique. Indeed, if $n$ could also be written as $n = \tilde{l}q^{\tilde{k}}$ for some prime $q \neq p$ with $\tilde{l} < q$ and $\tilde{k} \ge 1$, then unique factorization implies $q^{\tilde{k}} | l$ and $p^k | \tilde{l}$, from which we would get $l \ge q^{\tilde{k}} > \tilde{l} \ge p^k > l$; contradiction. In other words, if $n = \tilde{l}q^{\tilde{k}}$, then $\tilde{l} > q$, so that, in particular, $q^{\tilde{k}+1}$ must divide $L_n$. Let now $a_1, a_2 \in A$ be such that $a_1 \not \equiv a_2 \pmod{p}$. Then, regardless of whether we choose $r_n = a_1$ or $r_n = a_2$, for any $q < n$ different from $p$ we have $X_n = X_{n-1} + \frac{L_n r_n}{n} \equiv X_{n-1} \pmod{q}$, which we assumed to be non-zero for all $q < n$. On the other hand, $X_{n-1} + \frac{L_n a_1}{n} \not \equiv X_{n-1} + \frac{L_n a_2}{n} \pmod{p}$, so that at least one of those is non-zero modulo $p$. Set $r_n$ to an $a_i$ for which this holds, and $X_n = X_{n-1} + \frac{L_n r_n}{n} \not \equiv 0 \pmod{p}$. \\

Case III. For all $p < n$, writing $n = lp^k$ implies $l > p$. \\
As we noted in the previous case, this implies $e_p(L_n) \ge k + 1$. And so regardless of the value of $r_n$ we get $X_n = X_{n-1} + \frac{L_n r_n}{n} \equiv X_{n-1} \pmod{p}$ which is non-zero for all $p < n$, by the induction hypothesis. And we conclude that, for this case, we may choose $r_n$ arbitrarily. \\

In all cases it was possible for us to choose $r_n \in A$ in such a way that $X_n \not \equiv 0 \pmod{p}$ holds for all $p \le n$, and the theorem is proved.
\end{proof}

We therefore cannot give an upper bound on $b(a)$ that holds for all bounded sequences of $r_i$. On the other hand, the lower bound from Theorem \ref{limba} does still hold. Indeed, its proof does not require the $r_i$ to be periodic, and one can check that the $r_i$ are even allowed to grow a little as a function of $i$. Moreover, we claim that this lower bound is tight.

\begin{nonperiodiclower} \label{nonperiodiclower}
There exists a (non-periodic) sequence $r_1, r_2, \ldots$ with $r_i \in \{0, 1\}$ for all $i \in \mathbb{N}$, for which ${\displaystyle \liminf_{a \rightarrow \infty}} \left(\frac{b(a)-a}{\log a}\right) = \frac{1}{2}$.
\end{nonperiodiclower}

\begin{proof}
We will employ the same ideas and notation we used in Section \ref{lower2}, so familiarity with that section is assumed. For a quick reminder, recall that we set $b = xQ$ where $Q$ was defined as a product of primes $p$ for which a certain polynomial $f$ had a root $x_p$ modulo $p$, and $x$ was such that $x \equiv x_pQ_p^{-1} \pmod{p}$ where $Q_p = \frac{Q}{p}$. Moreover, there was one unique prime $q$ for which the root $x_q$ was specifically chosen, in order to obtain $x > Q_q$. \\

For this proof we are going to do the same thing, but by choosing $r_i = 0$ for most $i$ we can make sure that $f$ is, for all relevant primes $p \neq q$, a linear polynomial. This guarantees that it has a root modulo $p$. \\

Let $b_0$ be large enough so that for all $b \ge b_0$ there exists a prime $q \in (\sqrt{b}, \frac{b}{2}]$ for which $f_2(x) := 3x^2 - 6x + 2 \equiv 0 \pmod{q}$ is solvable.\footnote{Even though we did not mention this in Section \ref{lower2}, one can check by quadratic reciprocity that $f_2(x)$ has a root modulo an odd prime $q$ if, and only if, $q \equiv \pm 1 \pmod{12}$. By results in \mbox{\cite{prim12}} we can then deduce that $b_0 = 22$ would suffice.} If $b_{n-1}$ is defined for some $n \in \mathbb{N}$, then define $Q$ to be the product of all primes $p \in I_n := (\sqrt{b_{n-1}}, b_{n-1}]$, set $Q_p = \frac{Q}{p}$, and let $q_n$ be any prime in $(\sqrt{b_{n-1}}, \frac{b_{n-1}}{2}]$ for which $f_2(x) \equiv 0 \pmod{q_n}$ is solvable. For a prime $p \in I_n$ different from $q_n$, define $x_p = \frac{p+1}{2}$ and define $x$ as the largest integer smaller than $Q$ with $x \equiv x_pQ_p^{-1} \pmod{p}$ for all $p \in I_n$ different from $q_n$, and $f_2(xQ_{q_n}) \equiv 0 \pmod{q_n}$. Analogous to what we observed in Section \ref{lower2}, we have $x > Q_{q_n}$ since $f_2$ has two roots modulo $q_n$. \\

Now define $b_n = xQ$ and let $a_n$ be equal to $b_n - b_{n-1}$. This defines an infinite sequence of ever-growing $a_n$ and $b_n$, and one can check that PNT implies ${\displaystyle \lim_{n \rightarrow \infty}} \left(\frac{b_n - a_n}{\log a_n}\right) = \frac{1}{2}$. All we need to do is choose $r_i$ such that $b(a_n) \le b_n$ for all $n \in \mathbb{N}$. \\

For any positive integer $i$, choose $r_i = 0$, unless there exists an $n \in \mathbb{N}$ with either $i = b_n$, or $i = b_n - p$ for some $p \in I_n$, or $i = b_n - 2q_n$. Choose $r_i = 1$ in these latter three cases. We claim that indeed $b(a_n) \le b_n$ holds for all $n \in \mathbb{N}$. We will not repeat all details from Section \ref{lower2}, but essentially all we need to do, is check that $p$ divides $X_{a_n, b_n}$ for all $p \in I_n$. This is a consequence of the following congruence, where $d = 1$ for $p \neq q_n$ and $d = 2$ for $p = q_n$: 
\begin{equation*}
X_{a_n, b_n} \equiv \frac{L_{a_n,b_n}}{p} \sum_{i=0}^d \frac{1}{x_p - i} \pmod{p}
\end{equation*}

And this is congruent to $0 \pmod{p}$ by the construction of $x_p$. 
\end{proof}

We conclude that we can, in the non-periodic but bounded case, still prove a lower bound on $b(a)$, and that this lower bound is actually tight. There is however one other important result that we can generalize to the non-periodic case, and that is Theorem \ref{ohyah}. \\

To properly state this generalization, let $\{r_i\}_{i \in \mathbb{N}}$ be any bounded sequence of non-zero integers, with $r = \max_i |r_i|$. Let $m$ be any integer with $m > \max(3, r)$ and assume that there are $z$ primes strictly smaller than $m$. Moreover, let $\tilde{m}$ be any integer larger than $20m^{2z}$ such that $\tilde{m}$ has a prime divisor larger than $m^{2z-1}$, and define the interval $I = [\tilde{m} - m^{2z-1}, \tilde{m} + m^{2z-1})$. We can then state our generalization of Theorem \ref{ohyah}.

\begin{oldschool} \label{oldschool}
There exists an integer $n \in I$ for which $X_n$ is divisible by a prime larger than or equal to $m$. 
\end{oldschool}

Note that both $|I|$ and $\tilde{m}$ are slightly smaller than they were in Theorem \ref{ohyah}. This is due to the assumption $r_i \neq 0$ for all $i$, which guarantees that $\Sigma_3$ as defined in Section \ref{largeprimedivisors} is empty. Recall that, in the original proof of Theorem \ref{ohyah}, we needed $n_j$ to be congruent to $i_1 \pmod{t^3r_{i_1}^2}$ because of Lemma \ref{Lemma3c}. For the analogous proof of Theorem \ref{oldschool} however, Lemma \ref{Lemma3c} would no longer be relevant, since $\Sigma_3 = \emptyset$. \\

A natural follow-up question is now: how many of the $r_i$ have to be $0$ in order for Theorem \ref{oldschool} to become false? Or, moving even further astray, fix $m$ and let $|r_i| < m$ for all $i$. Furthermore assume that $\displaystyle \left(\sum_{i=1}^k \frac{r_i}{i}\right)^{-1}$ is an integer for all $k$ with $1 \le k \le n$. What is the largest possible subset $A$ of $\{1, 2, \ldots, n\}$ such that for all $i \in A$ we have $r_i \neq 0$? \\

For example, it is easy to check that $A$ can be the set of powers of two, with $r_1 = 1$ and $r_{2^k} = -1$ for all $k \ge 1$. This gives $|A| > c\log(n)$ for $m = 2$, but it seems likely that much better constructions are possible. However, these questions, interesting and tempting as they may be, do lead us away from the original subject of this paper. So for now we gladly pass these questions on to the next brave soul. 

\section{Final thoughts and remarks}\label{remarks}
It is not hard to show that for every $\epsilon \in (0, 1]$ we can improve the inequality $v_{a,b} < v_{a,b-1}$ from Corollary \ref{inf} to the slightly stronger $v_{a,b} < \epsilon v_{a,b-1}$. To prove this, first recall that we chose $M$ in Section \ref{dioph} equal to $\left\lfloor e^{2m + \frac{4m}{3\log(m)}} \right\rfloor$ to make sure that $l > M$ was either divisible by a prime $q \ge m$, or by a prime $q < m$ with $q^{e_q(l)} \ge m^2$. If we instead choose $M$ to be equal to $\left\lfloor e^{2\epsilon^{-1}m + \frac{3\epsilon^{-1}m}{\log(m)}} \right\rfloor$, then we claim that $l > M$ is either divisible by a prime $q \ge \epsilon^{-1}m$, or by a prime $q < \epsilon^{-1}m$ with $q^{e_q(l)} \ge \epsilon^{-1}m^2$. Along similar lines as the proof of Lemma \ref{primeorpower}:
\begin{align*}
\prod_{q < \epsilon^{-1}m} q^{\left\lfloor\frac{\log(\epsilon^{-1}m^2)}{\log(q)}\right\rfloor} &\le \prod_{q < \epsilon^{-1}m} \epsilon^{-1}m^2 \\
&< (\epsilon^{-1}m^2)^{\frac{\epsilon^{-1}m}{\log(\epsilon^{-1}m)} \left(1 + \frac{3}{2\log(m)}\right)} \\
&< e^{2\epsilon^{-1}m \left(1 + \frac{3}{2\log(m)}\right)}
\end{align*}

To find an explicit bound on the smallest $b$ such that $v_{a,b} < \epsilon v_{a,b-1}$, one can then go through the calculations from Section \ref{final} again, which results in the constant $c$ from Theorem \ref{god} increasing to $c = e^{e^{e^{\epsilon^{-1}m\left(4 + \frac{11}{\log(m)}\right)}}}$. \\


In fact, in the classical case where $r_i = 1$ for all $i$, we can use Linnik's Theorem to provide us with a prime $p$ that we can apply in Theorem \ref{Theorem1} to effectively get $\displaystyle \liminf_{b \rightarrow \infty} \frac{v_{a,b}}{v_{a,b-1}} = 0$. To see this, let $k_0 \in \mathbb{N}$ be arbitrary and let $p$ be the smallest prime congruent to $1 \pmod{2^{k_0}}$. By the current best known bound on Linnik's Theorem (see \mbox{\cite{lin2}}), we have $p < c_12^{5k_0}$ for some constant $c_1$. Moreover, by Wolstenholme's Theorem (or common sense), $p$ divides $X_n$ for $n = l = p-1$, while $\gcd(l, X_{a,b-1}) \le l2^{-k_0}$ by Lemma \ref{neverodd}. Applying the proof of Theorem \ref{Theorem1} we then obtain $v_{a,b} < 2^{-k_0} v_{a,b-1}$ with $b < c_1^22^{10k_0}a$. For the sake of clarity and completeness, let us formally state these two results. 

\begin{vabepsilongeneral}
For all $\epsilon \in (0, 1]$ there exists a constant $c_{\epsilon} := e^{e^{e^{\epsilon^{-1}m\left(4 + \frac{11}{\log(m)}\right)}}}$ such that for all $a \in \mathbb{N}$ there exists a $b < c_{\epsilon}a$ for which $v_{a,b} < \epsilon v_{a,b-1}$.
\end{vabepsilongeneral}

\begin{vabepsilonclassical}
If $r_i = 1$ for all $i$, then there is an absolute constant $K$ such that for all $\epsilon \in (0, 1]$ and all $a \in \mathbb{N}$ there exists a $b < K\epsilon^{-10}a$ for which $v_{a,b} < \epsilon v_{a,b-1}$.
\end{vabepsilonclassical}

However, all of these upper bounds seem far from the truth. It seems likely that the much stronger bound $b(a) = a + O(a^{\epsilon})$ holds, and plausibly even $b(a) = a + O\big(\log^k(a)\big)$ for some $k$ that may or may not depend on the sequence of $r_i$. But even in the classical case it is unclear what the correct upper bound should be. One can furthermore propose the same conjectures for the quantity $b_d(a)$ that we introduced in Section \ref{powerino}. Generalizing even further, let $f_1, f_2, \ldots$ and $g_1, g_2, \ldots$ be two periodic sequences of integer-valued polynomials, with $g_i(i) \neq 0$ for all $i \in \mathbb{N}$, and consider sums $\displaystyle \frac{u_{a,b}}{v_{a,b}} = \sum_{i=a}^b \frac{f_i(i)}{g_i(i)}$ with $\gcd(u_{a,b}, v_{a,b}) = 1$ and $v_{a,b}$ positive. When does there, for every fixed $a$, exist a $b$ such that $v_{a,b} < v_{a,b-1}$? If so, what is the least such $b = b(a)$? We should point out that, in this generality, counterexamples do exist. For example, $v_{1,b} = 1$ when $g_i(i)$ divides $f_i(i)$ for all $i$, while $v_{1,b} = b$ when $g_i(i) = f_i(i)i(i+1)$. \\

As for lower bounds, we showed ${\displaystyle\liminf_{a \rightarrow \infty}} \left(\frac{b(a)-a}{\log a}\right) \ge \frac{1}{2}$ for all periodic sequences of $r_i$. However, is this lower bound optimal? Or is it the case that the limit inferior is always strictly larger than $\frac{1}{2}$? If the latter is true, is there at least for every $\epsilon > 0$ a sequence of $r_i$ with ${\displaystyle\liminf_{a \rightarrow \infty}} \left(\frac{b(a)-a}{\log a}\right) < \frac{1}{2} + \epsilon$? Similarly, in Theorem \ref{nonperiodiclower} we showed that for bounded, non-periodic sequences it is possible that ${\displaystyle\liminf_{a \rightarrow \infty}} \left(\frac{b(a)-a}{\log a}\right)$ is exactly $\frac{1}{2}$. But can this also be realized with a sequence for which $r_i \neq 0$ for all $i$? \\

In the classical case we conjecture ${\displaystyle\liminf_{a \rightarrow \infty}} \left(\frac{b(a)-a}{\log a}\right) = \frac{1}{1+c}$, with $c$ defined in Section \ref{lower2}. We furthermore conjecture that the global minimum for the quotient $\frac{b(a)-a}{\log a}$ occurs at $a = 24968370984798709551283169$ with $b(a) = a + 31$ and $\frac{b(a)-a}{\log a} \approx 0.5300989$. With a computer we have checked up to $a = 10^{300000}$ and no examples with a smaller quotient were found. In fact, the largest $a$ for which $\frac{b(a)-a}{\log a}$ is smaller than $0.54$ seems to be $a \approx 5.5890852 \cdot 10^{3458}$ with $b(a) = a + 4300$. \\

On another note, it can be conjectured that $b(a-1) > b(a)$ happens infinitely often, which might not be too hard to prove when $r_1 = t = 1$, or perhaps even in general. Other questions also remain in the classical case. For example, it is still open if $\gcd(X_n, L_n) = 1$ holds for infinitely many $n$ or not. This is equivalent to asking whether there are infinitely many $n$ such that, if $l = l(p)$ is the first digit of $n$ in base $p$, we have the inequality $\displaystyle \sum_{i=1}^l \frac{1}{i} \not \equiv 0 \pmod{p}$ for all $p < n$. Lemma $2.4$ in \mbox{\cite[p. 71]{chi0}} shows that for every prime $p$ this inequality holds for at least $p - cp^{\frac{2}{3}}$ distinct $l < p$, where $c = \left(\frac{9}{8}\right)^{\frac{1}{3}}$. \\

We end with one final question for the classical case: is it true that the inequality $v_{1,n} < v_{1,n-1}$ holds if, and only if, $n$ does not divide $v_{1,n}$?\footnote{See \mbox{\cite{oeis2}} for the sequence of $n$ such that $n \nmid v_{1,n}$.} With a computer we have tried to look for counterexamples, but have not found any for $n < 10^6$. We can at least prove one direction: if $v_{1,n} < v_{1,n-1}$, then $n$ is not a divisor of $v_{1,n}$. To see this, first note that $v_{1,n} < v_{1,n-1}$ is not possible if $n$ is a prime power. So we may assume $L_n = L_{n-1}$, which implies that there is a prime $p$ with $e_p(X_{n-1}) < \min\big(e_p(X_n), e_p(L_n)\big)$. Since $X_n = X_{n-1} + \frac{L_n}{n}$, we deduce $e_p\left(X_{n-1} + \frac{L_n}{n}\right) > e_p(X_{n-1})$, which is only possible if $e_p(X_{n-1}) = e_p\left(\frac{L_n}{n}\right)$. And we then get $e_p(v_{1,n}) = \max(e_p(L_n) - e_p(X_n), 0) < e_p(L_n) - e_p(X_{n-1}) = e_p(L_n) - \big(e_p(L_n) - e_p(n)\big) = e_p(n)$, so that $n$ does not divide $v_{1,n}$. As for the other direction, if there exist primes $p, q$ with $p < q < p^2$ and such that, with $n = pq$, we have $e_p(X_{n-1}) = 1$, $e_p(X_n) \ge 2$ and $e_q(X_{n-1}) \ge 1$, then one can check $n \nmid v_{1,n}$ but $v_{1,n} > v_{1,n-1}$. However, we have not been able to find any such $n$, and it is unclear if they should exist.

\section{Acknowledgements}
The author is thankful for interesting and inspiring requests by Ernie Croot to make the results as general as possible. Without his questions these $14$ years of research would not have happened. The author is also indebted to Johannes Huisman and Anneroos Everts for their careful reading of the start of Section \ref{upper}. They managed to find a couple of small inaccuracies and their comments greatly improved the readability of the entire section.

\end{document}